\documentclass[12pt, reqno]{amsart}
\usepackage[utf8]{inputenc}
\usepackage{mathtools}
\usepackage{amsmath}
\usepackage{amssymb}
\usepackage{graphicx}
\usepackage{changepage}
\graphicspath{ {images/} }
\usepackage{mathrsfs}
\usepackage{listings}
\usepackage{amsthm}
\usepackage[toc,page]{appendix}
\usepackage[margin=0.9in]{geometry}
\usepackage{framed,enumitem}
\usepackage{amsaddr}
\usepackage{newtxtext}
\usepackage{bbm}

\usepackage[colorlinks = true,
linkcolor = blue,
urlcolor  = blue,
citecolor = blue,
anchorcolor = blue]{hyperref}

\usepackage{thmtools, thm-restate}

\newtheorem{thm}{Theorem}[section]
\newtheorem{defn}[thm]{Definition}
\newtheorem{cor}[thm]{Corollary}
\newtheorem{lem}[thm]{Lemma}
\theoremstyle{definition}
\newtheorem{prop}[thm]{Proposition}
\newtheorem{rk}[thm]{Remark}

\newcommand*{\Cdot}{{\raisebox{-0.5ex}{\scalebox{1.8}{$\cdot$}}}} 

\renewcommand{\H}{H}

\newcommand{\B}{X}

\newcommand{\ball}{B}

\newcommand{\K}{K}

\newcommand{\normH}[1]{\Vert #1 \Vert_{\H}}
\newcommand{\normB}[1]{\Vert #1 \Vert_{\B}}
\newcommand{\normBdual}[1]{|#1|}
\newcommand{\e}{\epsilon}
\newcommand{\R}{\mathbb{R}}
\renewcommand{\P}{\mathbb{P}}

\newcommand{\fo}{\quad f.o.}

\newcommand{\io}{\quad i.o.}

\newtheorem{example}[thm]{Example}

\newcommand{\Csp}{C}
\newcommand{\RR}{R}
\newcommand{\w}{w}
\newcommand{\Hsp}{H}
\newcommand{\h}{h}
\newcommand{\pa}{\partial}

\newcommand{\Cznorm}[1]{\|#1\|_{0}}

\newcommand{\normL}[1]{\|#1\|_{L^2}}

\renewcommand{\Im}{\text{Im}}
\newcommand{\red}[1]{\textcolor{red}{#1}}

\newcommand{\N}{\mathbb{N}}

\numberwithin{equation}{section}

\title{A nonlinear Strassen Law for singular SPDEs}
\author{Shalin Parekh}

\email[Shalin Parekh]{parekh@umd.edu}

\begin{document}

\begin{abstract}
A result of Arcones \cite{Arc95} implies that if a measure-preserving linear operator $S$ on an abstract Wiener space $(X,H,\mu)$ is strongly mixing, then the set of limit points of the random sequence $((2\log n)^{-1/2}S^nx)_{n\in\mathbb N}$ equals the closed unit ball of $H$ for a.e. $x \in X$, which may be seen as a generalization of the classical Strassen's law of the iterated logarithm. We extend this result to the case of a continuous parameter $n$ and higher Gaussian chaoses, and we also prove a contraction principle for Strassen laws of such chaoses under continuous maps. We then use these extensions to recover or prove Strassen-type laws for a broad collection of processes derived from a Gaussian measure, including ``nonlinear" Strassen laws for singular SPDEs such as the KPZ equation.
\end{abstract}

\maketitle

\tableofcontents 

\section{Introduction}

The law of the iterated logarithm for Brownian motion states that if $B$ is a standard Brownian motion then $\limsup_{t\to 0} (2t\log\log (1/t))^{-1/2}B_t = 1$. Strassen in a seminal work \cite{Str64} generalized this statement to show the functional form of this statement, namely that if we let $B^\epsilon(t) = \epsilon^{-1/2}B(\epsilon t)$ then the set of limit points as $\e \to 0$ in $C[0,1]$ of the sequence $\{(2\log\log(1/\epsilon))^{-1/2}B^\epsilon\}_\epsilon$ is almost surely equal to the closed unit ball of its Cameron-Martin space. Since Strassen's original work, there has been a tremendous effort resulting in a large literature expanding the scope of the theorem into many different settings including Banach space-valued processes \cite{KL73, Kue1, GKZ, Ale2}, Gaussian processes and higher chaoses \cite{Oo72, GK91, GK93, Arc95}, iterated processes \cite{Bur93, Arc95, CCFR95, Neu98}, stronger topologies \cite{PBK92}, sharper envelopes \cite{Rev79}, as well as more complicated stochastic processes driven by multiparameter fields and fractional processes \cite{Park75, DU99, CWZ04, MO05, Fa19}, etc. See \cite[Chapter 8]{LT96} for surveys of classical topics. 

The goal of the present paper is to extend Strassen's law in yet another general direction, related to many of the aforementioned extensions, eventually culminating with a compact limit set theorem for the small-noise regime of subcritical singular SPDEs such as the KPZ equation, which have been of popular interest in the probability literature recently, see e.g. \cite{Cor12, CW15, CS19}, etc. One difference from the aforementioned results is that we take a dynamical-systems and semigroup-theory perspective on Strassen's law, rather than considering an arbitrary sequence of random variables in a Banach space. We are uncertain if this perspective is novel, or if it merely provides a convenient notational framework to summarize proof methods that may already be well-known to experts in the area. Nonetheless the dynamical systems perspective does allow us to concisely recover and generalize some of the more classical results cited above, in particular our main results (see Theorem \ref{mainr} and its generalizations Theorems \ref{gen}, and \ref{repar}) recover results from \cite{KL73, Park75, PBK92, Bur93,GK93,CCFR95,Arc95, Neu98, DU99, CWZ04}, see Section \ref{rats}.

First we establish some notation. If $\H$ is any real Hilbert space and $S:\H \to \H$ is any bounded operator, we denote by $S^*$ its adjoint operator with respect to the inner product of $\H$. Throughout this work we will use the notion of abstract Wiener spaces introduced by Leonard Gross \cite{Gro}. These are formal triples $(X,H,\mu)$ where $X$ is a Banach space, $\mu$ is a centered Gaussian measure on $X$, and $H\hookrightarrow X$ is the compactly embedded Cameron-Martin space.
In this paper we will consider measure-preserving a.e. linear maps on $X$. These are Borel-measurable linear maps $S:E\to X$ where $E\subset X$ is a Borel-measurable linear subspace of full measure in $X$, such that $S$ preserves the measure $\mu$. It is known that any such map necessarily maps $H$ to itself and must satisfy $SS^*=I$ there, see Lemma \ref{bij}. 

Conversely, given any bounded linear map $S:\H \to \H$ satisfying $SS^*=I$, there exists a Borel-measurable linear subspace $E\subset X$ of full measure and a linear extension $\hat S:E \to \B$ of $S$ which is unique up to a.e. equivalence (this is despite the fact that $H$ may have measure zero). Moreover the condition $SS^*=I$ guarantees that this extension is measure-preserving. We refer to Lemma \ref{bij} for a proof of these statements. Letting $E_2:=\{x \in E: \hat Sx \in E\}$, we see that $E_2$ is a measurable linear subspace and we can sensibly define $\hat S^2x:=\hat S(\hat Sx)$ for $x\in E_2$. Moreover, the measure-preserving property implies that $\mu(E_2)=1$. Continuing iteratively we may define $E_{k+1} = \{x: \hat Sx\in E_k\},$ which will have full measure and admits a sensible definition of $\hat S^{k+1}$. Then $E_\infty:= \bigcap E_k$ will have full measure and admits a simultaneous definition of all $\hat S^n$. 

In the sequel we will not distinguish between $\hat S$ and $S$ and simply write $\hat S^n=:S_n$ without explicitly specifying these formalities of its definition on a linear subspace of full measure. The forward implication of the following result may be seen as a special case of \cite[Theorem 2.1]{Arc95} (albeit formulated differently) and the main focus of the present paper will be to extend it to continuous-parameter settings and higher Gaussian chaoses, ultimately culminating in a nonlinear version of Strassen's Law for some singular SPDEs. The reverse implication has not appeared in the literature as far as we know.

\begin{prop}\label{thm1} Let $(\B,\H,\mu)$ be an abstract Wiener space. Let $S: H \to H$ be a bounded linear operator and write $S_N:=S^N$ as explained above. Suppose that $SS^* = I$, i.e., $S$ is measure-preserving.

If $\langle S_Nx,y\rangle_\H\to 0$ as $N\to \infty$ for all $x,y \in H$ (i.e., $S$ is strongly mixing) then for $\mu$-almost every $x \in \B$, the set of limit points of $\{\frac{S_Nx}{\sqrt{2\log N}}: N \in \mathbb N\}$ is equal to the closed unit ball of $\H.$

Conversely, if the set of limit points of $\{\frac{S_Nx}{\sqrt{2\log N}}: N \in \mathbb N\}$ is a.s. equal to the closed unit ball of $\H,$ then one necessarily has $\frac1N\sum_{k=1}^N |\langle S_k x,y\rangle_H| \to 0$ for all $x,y\in H$ (i.e., $S$ is weakly mixing).
\end{prop}

This is proved as Proposition \ref{mr2.1} in the main body. Just to be clear, we are referring to limit points with respect to the topology of $\B$, and $S_Nx$ is meant to be understood in terms of the unique measurable linear extensions (explained above) if $x \in \B$. The equivalence of strong mixing or weak mixing to the stated conditions on the inner products $\langle S_kx,y\rangle$ is proved in Proposition \ref{prop1} below. It turns out that in many cases of interest, an even stronger mixing condition $\| S_Nx\|_\H\to 0$ is satisfied for all $x\in H$, and by Lemma \ref{easy} this can be equivalently stated as $\bigcap_{N \in \N} \Im(S_N^*)=\{0\}$, where $\Im(S)$ denotes the image in $H$ of $S$. This is easier to check in certain instances. Lemma \ref{sg'} provides another efficient way to check that the mixing condition is satisfied in many examples of interest.

Now we will formulate a continuous-time version of Proposition \ref{thm1}. Recall that a strongly continuous semigroup on a Banach space $\B$ is a semigroup $(S_t)_{t \geq 0}$ of bounded operators from $\B\to \B$ such that $\|S_tx-x\|_{\B}\to 0$ as $t \to 0$ for all $x\in \B$. Moreover, if $\B$ is a Banach space and if $\gamma:[0,\infty) \to \B$ is any function, then we call $x\in\B$ a \textit{cluster point at infinity} of $\gamma$ if there exists a sequence $t_n\uparrow \infty$ such that $\|\gamma(t_n)-x\|_{\B}\to 0$. Similarly if $\gamma:(0,1] \to \B$ then we will call $x\in\B$ a \textit{cluster point at zero} of $\gamma$ if there exists a sequence $\e_n\downarrow 0$ such that $\|\gamma(\e_n)-x\|_{\B}\to 0$.

\begin{prop}\label{thm 2}
Let $(\B,\H,\mu)$ be an abstract Wiener space. Suppose that $(S_t)_{t \ge 0}$ is a family of bounded operators from $H\to H$ satisfying the following four properties:

\begin{enumerate}[leftmargin = 2em]
    \item $S_{t+u}=S_tS_u$ for all $t,u \ge 0$.
    
    \item $S_tS_t^*=I$.
    
    \item $\langle S_tx,y\rangle_H\to 0$ as $t\to \infty$ for all $x,y\in \H$.
    
    \item $(S_t)_{t\ge 0}$ extends to a strongly continuous semigroup on the larger space $\B.$
\end{enumerate}
Then for $\mu$-almost every $x \in \B$, the set of cluster points at infinity of $\{\frac{S_tx}{\sqrt{2\log t}}: t \geq 0\}$ is equal to the closed unit ball of $\H.$

\end{prop}

This will be proved by combining the previous proposition with Lemma \ref{glue} in the main body. Now let us relate this to the usual law of the iterated logarithm. We will say that a strongly continuous \textit{multiplicative} semigroup is a family of bounded operators $(R_{\epsilon})_{\epsilon \in (0,1]}$ satisfying $R_{\epsilon}R_{\delta} = R_{\epsilon\delta}$ and moreover $\|R_{\epsilon}x-x\|_{\B} \to 0$ as $\epsilon \uparrow 1$ for all $x\in\B$. Then by setting $S_t:=R_{e^{-t}}$, Proposition \ref{thm 2} can be restated in terms of multiplicative semigroups as follows.
\\
\\
Let $(\B,\H,\mu)$ be an abstract Wiener space. Suppose that $(R_\epsilon)_{\epsilon\in (0,1]}$ is a family of bounded operators from $H \to H$ satisfying the following four properties:

\begin{enumerate}
    \item $R_{\epsilon}R_{\delta} = R_{\epsilon\delta}$ for all $\epsilon,\delta\in (0,1]$.
    
    \item $R_{\epsilon}R_{\epsilon}^*=I$.
    
    \item $\langle R_{\epsilon}x,y\rangle_H\to 0$ as $\epsilon \to 0$ for all $x,y\in \H$.
    
    \item $(R_{\epsilon})_{\epsilon\in (0,1]}$ extends to a strongly continuous multiplicative semigroup on the larger space $\B.$
\end{enumerate}
Then for $\mu$-almost every $x \in \B$, the set of cluster points at zero of $\{\frac{R_{\epsilon}x}{\sqrt{2\log \log (1/\epsilon)}}: \epsilon\in (0, 1]\}$ is equal to the closed unit ball of $\H.$
\\
\\
The quintessential example of such a setup is the classical Wiener space $\B=C[0,1], \H=H_0^1[0,1]$, and $\mu$ is the law of a standard Brownian motion. Then one sets $R_{\epsilon}f(x) = \epsilon^{-1/2}f(\epsilon x).$ One easily checks that $\|R_\e f\|_{H_0^1} \to 0$ as $\e\to 0$ for all $f\in H_0^1[0,1],$ and thus the above statement recovers the classical version of Strassen's law. Another important example in the context of stochastic PDEs is the additive-noise stochastic heat equation $h$ driven by space-time white noise, started either from two-sided Brownian motion or from zero initial condition, in which one obtains a compact limit set theorem for the $\e$-indexed family of functions $\epsilon^{-1/2} h(\epsilon^2 t, \epsilon x)$, see Example \ref{additive she} below. We will give a self-contained proof of Propositions \ref{thm1} and \ref{thm 2} in Subsection \ref{sub_mr} below.

Next let us discuss our main result of the paper, the generalization to higher chaos and the contraction principle. If $(\B,\H,\mu)$ is an abstract Wiener space, then the $n^{th}$ homogeneous Wiener chaos, denoted by $\mathcal H^k(\B,\mu)$ is defined to be the closure in $L^2(\B,\mu)$ of the linear span of $H_k \circ g$ as $g$ varies through all elements of the continuous dual space $X^*,$ where $H_k$ denotes the $k^{th}$ Hermite polynomial (normalized so that $H_k(Z)$ has unit variance for a standard normal $Z$). Letting $Y$ be another separable Banach space, a Borel-measurable map $T:\B\to Y$ is called a chaos of order $k$ over $X$ if $f\circ T \in \mathcal H^k(\B,\mu)$ for all $f \in Y^*$. For such a chaos we may define (following \cite{Led96, HW13}) its ``homogeneous form" $T_{\mathrm{hom}} : H \to Y_i$ by the Bochner integral formula $$T_{\mathrm{hom}}(h):= \int_X T(x+h)\mu(dx) = \frac1{k!} \int_X T(x) \langle x,h\rangle^k\mu(dx),$$ which may be shown to converge, see Proposition \ref{eq'} below which is based on work of \cite{AG93}. 

For an abstract Wiener space $(X,H,\mu)$ we always let $B(H)$ denote the closed unit ball of the Cameron-Martin space $H$. Our main result in this work is the following.

\begin{thm}\label{mainr}
Let $(\B,\H,\mu)$ be an abstract Wiener space, and let $(R_\epsilon)_{\epsilon\in (0,1]}$ be a family of Borel-measurable a.e. linear maps from $\B\to \B$ which are measure-preserving and satisfy $\langle R_\e x,y\rangle_H \to 0$ as $\e \to 0$ for all $x,y\in H$ (i.e., each $R_\e$ is strongly mixing). Let $T^i:\B\to Y_i$ be a chaos of degree $k_i$ over $X$ for $1\le i \le m$, where $Y_i$ may be any separable Banach spaces. Suppose that there exist strongly continuous semigroups $(Q^i_\epsilon)_{\epsilon\in (0,1]}$ of operators from $Y_i\to Y_i$ for $1\le i \le m$ such that 
\begin{equation}\label{commute}T^i\circ R_\epsilon = Q^i_\epsilon \circ T^i, \;\;\;\;\;\;\;\; \mu\text{-a.e.}\;\;\;\;\;\; \text{ for all } \e\in (0,1] \text{ and } 1\le i \le m.
\end{equation}
Let $\mathcal M\subset Y_1\times \cdots Y_m$ be a closed subset, such that for all $\delta>0$  $$\mu(\{x\in X: (\delta^{k_1} T^1(x),...,\delta^{k_m} T^m(x)) \in \mathcal M\})=1.$$
Then the compact set $K:=\{(T^1_{\mathrm{hom}}(h),...,T^m_{\mathrm{hom}}(h)):h \in B(H)\}$ is contained in $\mathcal M$. Moreover for any Banach space $Z$ and any continuous map $\Phi:\mathcal M \to Z$, the set of cluster points at zero of
$$\big\{ \Phi\big( (2\log \log(1/\e))^{-k_1/2} Q_\e^1T^1 (x),\;.\;.\;.\;,(2\log \log(1/\e))^{-k_m/2}Q_\e^mT^m (x)\big): \e \in (0,1]\}$$ is almost surely equal to $\Phi(K).$
\end{thm}

This result builds upon a body of literature on sample paths of ergodic self-similar processes derived from a Gaussian measure and their chaoses, which started with works such as \cite{ooo, ppp, qqq}. Besides the more abstract setup, a novelty of our result is that there is no additional difficulty in allowing the topology on the spaces $Y_i$ and $Z$ to be taken as strong as one desires (e.g. Hölder spaces as opposed to just continuous functions).

Theorem \ref{mainr} is proved as Corollary \ref{contraction 2} in the main body of the paper. Note that one recovers Proposition \ref{thm 2} by setting $m=1,k_1=1, Y_1=X=Z=\mathcal M$, and making $\Phi$ the identity map. Classical and new examples using this general result are illustrated in Sections 3 and 4. One particular example of interest which seems new, and which sparked our interest in this problem, is the following.

\begin{thm}\label{buss} For $\e \in (0,1/e)$ let $C_\e := (\log\log(1/\e))^{1/2}$, and let $h^\e$ denote the Hopf-Cole solution to the KPZ equation on $\mathbb R_+\times \mathbb R$ driven by standard Gaussian space-time white noise: $$\partial_t h^\e = \partial_x^2 h^\e + C_\e (\partial_x h^\e)^2 + \xi,\;\;\;\;\;\;\;\;\;\;t\ge 0, x\in \mathbb R,$$ with initial data $h(0,x) = 0$. Then for any $s,y> 0$ the set of cluster points as $\e \downarrow 0$ in $\mathcal C_{s,y}:=C([0,s]\times [-y,y]) $ of the sequence of functions $C_\e^{-1}\e^{-1/2} h^\e (\epsilon^2 t, \epsilon x)$ is a.s. equal to the compact set $K_{\mathrm{Zero}}$ given by the closure in $\mathcal C_{s,y}$ of the set of smooth functions $h$ satisfying $$h(0,x) = 0,\;\;\;\;\; \big\| \partial_t h - \partial_x^2 h - (\partial_x h)^2 \big\|_{L^2([0,s]\times [-y,y])}\leq 1.$$
Moreover, the same compact limit set result holds for the same sequence of functions in stronger topologies given by parabolic Hölder norms up to but excluding exponent 1/2 (see \eqref{holderspace}).
\end{thm}

This will be proved in Section 3.2. We remark that the compact limit set is precisely the compact set on which the large deviation rate function of small-noise KPZ is less than or equal to 1, see for instance \cite{HW13, Lin}.

The result of Theorem \ref{mainr} can also be used to prove variants of Theorem \ref{buss}. For instance if we instead let $h^\e(0,x)$ be a two sided Brownian motion (fixed for different values of $\e)$ then one could show that a similar result holds for $C_\e^{-1}\e^{-1/2} h^\e (\epsilon^2 t, \epsilon x)$, but with compact limit set $K_{Br}$ given by the closure of smooth functions $h \in \mathcal C_{s,y}$ satisfying $$h(0,0)=0,\;\;\;\;\; \|\partial_x h(0,\cdot)\|_{L^2[-y,y]}^2+\big\| \partial_t h - \partial_x^2 h - (\partial_x h)^2 \big\|^2_{L^2([0,s]\times [-y,y])}\leq 1.$$If we likewise define $k^\e$ to be the solution of $$\partial_t k^\e = \partial_x^2 k^\e + C_\e^{-1}(\partial_x k^\e)^2 +\xi,$$ then one can further show that the same compact limit set results hold for $C_\e^{-1}\epsilon^{1/2}k^\e(\epsilon^{-2}t,\epsilon^{-1}x)$, in other words the large-time regime. For brevity we will not write down the proofs of these variants.

Note that in Theorem \ref{buss}, the nonlinearity in the equation describing $h^\e$ must be scaled along with the parameter $\e$, so that it blows up in the $\e\to 0$ limit. If we did not do this, then the limiting compact set would simply agree with that of the linearized equation $\partial_t h_{\mathrm{Linear}} = \partial_x^2 h_{\mathrm{Linear}} +\xi$, namely $$K_{\mathrm{Linear}}:= \{h \in \mathcal C_{s,y}: h(0,x)=0, \big\| \partial_t h - \partial_x^2 h  \big\|_{L^2([0,s]\times [-y,y])}\leq 1\}.$$
Indeed this can be proved by decomposing $h_{\mathrm{KPZ}} = h_{\mathrm{Linear}} + v$ where $h_{\mathrm{KPZ}}$ is the Hopf-Cole solution to KPZ with initial data zero and $v$ is a remainder term which has better regularity than $h_{\mathrm{Linear}}$ (see e.g. Theorem 3.19 of \cite{PR19}). Then under the scaling necessary to obtain Strassen's law, it is easy to check that the remainder term converges a.s. to zero in the topology of $C([0,s]\times [-y,y])$ and the set of limit points for the part corresponding to $h_{\mathrm{Linear}}$ can be shown to be $K_{\mathrm{Linear}}$ by applying Proposition \ref{thm 2} above (see Example \ref{additive she} below). Likewise in the result stated above for the family $k^\e$, the nonlinearity must be scaled along with the parameter $\e$ so that it vanishes in the $\e \to 0$ limit. If we did not do this, then the asymptotics would be wrong entirely and instead one would need to apply a scaling that respects the tail behavior of the so-called \textit{KPZ fixed point} \cite{MQR17}, namely $(\log\log(1/\e))^{2/3}$, see \cite{DG19}. The functional form of the latter result seems to be a very interesting open problem, see \cite{TQ} for some progress on large deviations under the KPZ fixed point scaling. 

The manner in which we prove Theorem \ref{buss} uses results from the theory of regularity structures \cite{Hai13} and is robust enough to prove similar theorems for other rough equations on a full space that admit a factorization of the solution map in terms of a finite number of Gaussian chaoses and a continuous part. In the case of $\Phi_2^4$ and $\Phi_3^4$, such a methodology would only be applicable once one has developed such a solution theory for the latter equation on the full space $\mathbb R^2$ or $\mathbb R^3$ using regularity structures. To show the generality of our method, we also prove a similar result for mollified smooth versions of the noise, specifically we show that even though mollifying is not a measure-preserving operation on the noise, it is enough to be ``asymptotically rapidly measure-preserving," see Theorem \ref{contraction 3} and its corollary Theorem \ref{mollkpz} for the KPZ equation. 

See Theorem \ref{gen} and the subsequent discussion for the most general formulation of the main theorem above, in which we discuss how one could potentially use our results to give Strassen laws even for discrete or non-Gaussian systems, assuming that a strong enough form of the ``almost sure invariance principle" holds for the non-Gaussian driving noise as well as for its higher chaoses.

In Section 4.2 we consider yet another generalization of the main result that may be applied to obtain Strassen laws for so-called ``iterated processes," recovering and generalizing \cite{Bur93, CCFR95, Neu98}.
\\
\\ 
\textbf{Organization.} In Section 2 we prove the main result Theorem \ref{mainr}, first expositing on the known results before moving onto the newer higher chaos results. Section 3 is devoted to giving many examples culminating in the proof of Theorem \ref{buss}. Section 4 gives further generalizations of Theorem \ref{mainr}. Appendix A studies the ergodic properties of measure-preserving linear operators on an abstract Wiener space. Appendix B contains functional-analytic lemmas that are useful in checking the conditions of Theorem \ref{mainr} in cases of interest. 
\\
\\
\textbf{Notation and Conventions.} Throughout this paper we use $(X,H,\mu)$ to denote an abstract Wiener space, and we use $x$ or $\xi$ to denote a general element of the space $X$. We generally use $S_N,G_t,R_\e,Q_\e$ for measure-preserving semigroups of operators on these spaces. For a Gaussian chaos we typically use $T$ or sometimes $\psi$, and for its homogeneous form we always use the subscript ``$\mathrm{hom}$."
\\
\\
\textbf{Acknowledgements.} We thank Yier Lin for many fruitful discussions on the topic of large deviations of Gaussian chaos. This research was partially supported by Ivan Corwin's NSF grant DMS-1811143, and the Fernholz Foundation's ``Summer Minerva Fellows'' program.

\section{Proofs of main theorems}

\subsection{Strassen law for strongly mixing linear operators}\label{sub_mr}

In this subsection we prove Propositions \ref{thm1} and \ref{thm 2}. Although most of these results for first-order chaos are already known (as we explained in the introduction), we give a short proof here because we use some of the lemmas later for the higher chaos results, and because it allows the exposition to flow more smoothly in later sections. First we need some preliminary lemmas and then we will formulate the result as Theorem \ref{mr2.1}.

\begin{lem}\label{lem:lim1}
Let $(X_i)_{i\ge 0}$ be a stationary sequence of real-valued jointly Gaussian random variables. Suppose that $\mathrm{var}(X_0)=1$ and $\mathrm{cov}(X_0,X_n)\to 0$ as $n\to \infty$. Then $$\limsup_{n\to \infty} \frac{X_n}{\sqrt{2\log n}} = 1,\;\;\;\;\;\;\;a.s..$$
\end{lem}

This lemma is classical and is a special case of \cite[Lemma 2.1]{Arc95}, which in turn is an improvement of classical results from \cite{pick, Nis67, Lai73}. The proof may be found in those works, or proved directly as a straightforward corollary of Slepian's Lemma (using a comparison with the stationary Gaussian process $\xi_j$ with unit variance, whose covariance is given for $i<j$ by $\mathrm{cov}(\xi_i,\xi_j)=\epsilon$ for some fixed small $\epsilon \in (0,1)$).

We remark that the above lemma is the only place where the mixing condition is used, and is also the reason why we are not able to extend the result to the ergodic case (i.e., it seems difficult to prove or find a counterexample to the above fact under the weaker assumption $\frac1{n} \sum_1^n |\mathrm{cov}(X_0,X_n)|\to 0$).
Next we generalize the above lemma to $\mathbb R^N$. We denote the unit sphere of $\mathbb R^N$ to be the set of all points $(a_1,...,a_N)$ with $\sum_1^N a_i^2=1.$

\begin{cor}\label{cor:sphere}
Let $(\vec X_n)_{n\ge 1}$ be a stationary sequence of jointly Gaussian random variables in $\mathbb R^N$, say $\vec X_n = (X_n^1,X_n^2,...,X_n^N).$ Suppose that $\mathrm{cov}(X_0^i,X_0^j)=\delta_{ij}$ and that $\mathrm{cov}(X_0^i, X_n^j) \to 0$ as $n\to \infty$ for all $1\le i,j\le N$. Then the unit sphere of $\mathbb R^N$ is contained in the set of limit points of the random sequence $((2\log n)^{-1/2}\vec X_n)_{n\ge 2}.$
\end{cor}

\begin{proof}
By Lemma \ref{lem:lim1}, if we restrict our attention to only the first coordinate, then the set of limit points of the random sequence $((2\log n)^{-1/2}\vec X_n)_{n\ge 1}$ must contain a point on the set $A:=\{(a_1,...,a_N):a_1=1\}$. On the other hand, by Borel-Cantelli and $\mathrm{cov}(X_0^i,X_0^j)=\delta_{ij}$, it is clear that the set of limit points must be contained in the set $B:=\{(a_1,...,a_N): \sum_{i= 1}^N a_i^2 \le 1\}.$ 

Since $A\cap B = \{(1,0,...,0)\}$ it follows that $(1,0,...,0)$ is a.s. a limit point of the random sequence $((2\log n)^{-1/2}\vec X_n)_{n\ge 1}.$ The fact that any point on the unit sphere is a limit point then follows from rotational invariance of the conditions that $\mathrm{cov}(X_0^i,X_0^j)=\delta_{ij}$ and $\mathrm{cov}(X_0^i, X_n^j) \to 0$ (i.e., these conditions remain true if we replace $(\vec X_n)_n$ by $(U(\vec X_n))_n$ for some orthogonal $N\times N$ matrix $U$).
\end{proof}

\begin{defn}
    Let $H$ be a real Hilbert space. We will henceforth denote $S(H):= \{h \in \H: \|h\|_H=1\}$ and $B(H):=\{h\in\H: \|h\|_H\le 1\}.$
\end{defn}

\begin{lem}\label{lem:s=b}
Let $(\B,\H,\mu)$ be an abstract Wiener space of infinite dimension. Then the closure in $\B$ of $S(H)$ is $B(H).$
\end{lem}

\begin{proof}
$B(H)$ is a compact (hence closed) subset of $\B$ which contains $S(H)$, so the set of limit points of $S(H)$ must be contained in $B(H)$. Choose an orthonormal basis $\{e_n\}_n$ for $\H$. Since $e_n \in B(\H)$, since $B(\H)$ is compact in $\B$ \cite{Bog}, and since $e_n\to 0$ weakly in $\H$, it follows that $\|e_n\|_\B\to 0$ as $n\to \infty$. Now fix $h \in H$ with $\|h\|_H\le 1$. Choose $c_n\in \mathbb R$ such that $\|h+c_ne_n\|_H=1$. It is clear that $|c_n|\leq 2$ (otherwise $1=\|c_ne_n+h\|_H \geq |c_n|\|e_n\|-\|h\| >2-1=1  $), and therefore $\|c_ne_n\|_\B\to 0$. Thus $h+c_ne_n$ is a sequence in $S(H)$ converging to $h\in B(H)$ with respect to the topology of $\B.$
\end{proof}

At this point we invite the reader to review the main results of Appendix \ref{AA}. Given an abstract Wiener space $(X,H,\mu)$, the set of measure-preserving a.e. linear maps from $X\to X$ (modulo a.e. equivalence) is in bijection with the set of bounded linear maps from $H\to H$ satisfying $SS^*=I$. Moreover, the bijection is given by simply restricting $S$ to $H$. This is shown in the appendix, see Lemma \ref{bij}. In the sequel we will not notationally distinguish between such a map $S$ on $H$ and its a.e.-unique and a.e.-linear measure-preserving extension on $X$.

\begin{lem}\label{lem:part1} 

Let $(\B,H,\mu)$ be an abstract Wiener space. Suppose that $S:H\to H$ is a bounded linear operator satisfying $SS^*=I$. Then for any $a\notin B(H)$ there exists $\epsilon>0$ such that $$\mu \bigg(\bigg\{x\in \B: \frac{S_k x}{\sqrt{\log k}} \in \ball(a,\epsilon)\;\; \text{infinitely  often}\bigg\}\bigg)=0.$$

\end{lem}

\begin{proof}
By definition of the Cameron-Martin space, 
for every $x \in \B$, we have $$\sup \{\ell(x): \ell \in \B^*, \int \ell^2 d\mu = 1\} = \normH{x}$$ (we set $\|x\|_H = \infty$ when $x \notin \H$). For any $x \in \B\backslash B(H)$,  there exists $\ell \in X^*$ such that $\ell(x) > \sqrt{2}$ and $\int_X \ell^2 d\mu = 1$. For each $n \in \N$, under $\mu$, 
$\ell(S_n x)$ is a standard Gaussian random variables. Using Borel-Cantelli lemma and Gaussian tail bound, together with $\ell(x) > \sqrt{2}$, there exist small enough $\e_x > 0$ such that
\begin{equation*}
\P\Big(\big\{\frac{\ell(S_n \xi)}{\sqrt{\log n}}\big\}_{n=1}^\infty  \in \ball(\ell(x), \e_x ) \;\; \text{infinitely  often} \Big) = 0
\end{equation*} 
By making $\e_x$ smaller, we have 
\begin{equation*}
\P\Big(\big\{\frac{S_n \xi}{\sqrt{ \log n}}\big\}_{n=1}^\infty  \in \ball(x, \e_x ) \;\; \text{infinitely  often} \Big) = 0. \qedhere
\end{equation*} 

\end{proof}

\begin{lem}\cite[Equation (4.4)]{Led96}\label{ledoux}
Let $\mu$ be a centered Gaussian measure on a separable Banach space, 
then for all $a>0$
\begin{equation*}
\mu( \{x\in X: \|x\|_X > a+ \int_X \|u\|_X \mu(du) \}) \leq e^{-a^2/(2\sigma^2)},
\end{equation*}
where $\sigma:= \sup_{\|f\|_{X^*} \leq 1} \int_X f(x)^2\mu(dx) = \sup_{\|h\|_H\leq 1} \|h\|_X.$
\end{lem}

Finally we are ready to prove the main result of the section, a restatement of Proposition \ref{thm1} in addition to a partial converse.

\begin{prop}\label{mr2.1}
Let $(\B,\H,\mu)$ be an abstract Wiener space. Let $E\subset \B$ be a Borel measurable linear subspace of measure 1, and suppose $S:E\to \B$ is linear and measure-preserving. If $S$ is strongly mixing, then the set of limit points of the random sequence $((2\log n)^{-1/2}S^nx)_{n\in\mathbb N}$ equals the closed unit ball of $H$ for a.e. $x \in X$. 

Conversely, if $S$ is any measure-preserving linear operator such that the set of limit points of the random sequence $((2\log n)^{-1/2}S^nx)_{n\in\mathbb N}$ equals the closed unit ball of $H$ for a.e. $x \in X$, then $S$ must be ergodic (equivalently weakly mixing by Proposition \ref{prop1}).
\end{prop}

\begin{proof}
Fix $h \in \H$ with $\|h\|_H=1$, and let $\epsilon>0$. We wish to show that $\|(2\log n)^{-1/2}S_nx-h\|_H<\epsilon$ infinitely often. To this end, let $\{e_i\}_i$ be an orthonormal basis of $H$ with $e_1=h$. Let $P_k$ denote the orthogonal projection onto the subspace $M_k$ which is spanned by $\{e_i\}_{i=1}^k.$

Choose some $k\in \mathbb N$ so that $\int_\B\|x-P_kx\|_\B^2\mu(dx)<(\epsilon/5)^2.$ Then for all $n$ we have that 
\begin{equation}\label{est}\mathbb P(\|S_n\xi - P_k(S_n\xi)\|_X>a) = \mathbb P(\|\xi - P_k(\xi)\|_X>a) \leq e^{-Ma^2} \mathbb E[e^{M\|\xi- P_k(\xi)\|_X^2}]
\end{equation}
for any $a>0$ and for any $M>0$ such that the expectation on the right side is finite. By Lemma \ref{ledoux}, it is true that $\int_X e^{\alpha \|x\|^2} \nu(dx)$ is finite whenever $\alpha < (2\int_X \|x\|^2 \nu(dx))^{-1},$ for any centered Gaussian measure $\nu$ on $X.$ Consequently in \eqref{est} we can take $M$ to be $(5/\e)^2$ and we can take $a$ to be $\e \sqrt{2\log n}/4$ and we find
$$\sum_{n >1} \mathbb P\bigg( \frac{\|S_n\xi -P_k(S_n\xi)\|_\B}{\sqrt{2\log n}} >\epsilon/4\bigg)<\infty,$$ where $\xi$ is distributed according to $\mu.$ Then by the Borel-Cantelli lemma, we find that 
\begin{equation*}
(2\log n)^{-1/2} \|S_n\xi - P_k(S_n\xi)\|_\B<\epsilon/4 
\end{equation*}
for all but finitely many $n$, a.s.. Thus we just need to show that $\|(2\log n)^{-1/2}P_k(S_nx)-h\|_\B<\epsilon/4$ infinitely often. But this is just a finite dimensional statement which immediately follows from Corollary \ref{cor:sphere}. Indeed, by assumption $h$ lies on the unit sphere of $M_k$, and the joint covariances tend to zero precisely because of the condition that $S$ is mixing (via Item \textit{(2)} in Proposition \ref{prop1}).

In the notation of Lemma \ref{lem:s=b}, we have shown that any point on $S(H)$ is almost surely a limit point of $(2\log n)^{-1/2}S_nx$. By that same lemma, any point of $B(H)$ is also a limit point. It is clear from Lemma \ref{lem:part1} that no point outside of $B(H)$ can be a limit point of $(2\log n)^{-1/2}S_nx$. That lemma is still valid in this case, since it only relies on the measure-preserving property of $S_n$ and nothing else. This completes the proof of the first statement.

Next we prove that $S$ must be ergodic if the set of limit points of $((2\log n)^{-1/2}S^nx)_{n\in\mathbb N}$ equals the closed unit ball of $H$ for a.e. $x \in X$. Indeed, if $S$ is not ergodic, then by Item \textit{(3d)} in Proposition \ref{prop1}, there is a two-dimensional invariant subspace $M$ on which $S$ acts by a rotation matrix. We claim that no nonzero point which lies in $M$ can be visited infinitely often, since $(2\log n)^{-1/2}S_nx$ converges to zero a.s. for any $x\in M$. Indeed, let $P$ be the projection onto $M$ and let $Q=I-P$ denote the projection onto $M^{\perp}.$ Let $\xi$ be sampled from $\mu$, so that $Q\xi$ has Cameron-Martin space $M^{\perp}$. Since $M$ and $M^{\perp}$  are invariant under $S$, it follows that $PS_n=S_nP$ and $QS_n = S_nQ$. By Lemma \ref{lem:part1} it is clear that the set of limit points of $(2 \log n)^{-1/2}QS_n\xi = (2\log n)^{-1/2}S_n(Q\xi)$ must be contained in the closed unit ball of the Cameron-Martin space of $Q\xi$, namely $M^{\perp}$ (this lemma only relies on the measure-preserving property of $S_n$ and nothing else). Furthermore, since $(2\log n)^{-1/2}S_nx$ converges to zero a.s. for any $x\in M$, it follows that $(2 \log n)^{-1/2}PS_n\xi = (2\log n)^{-1/2}S_n(P\xi)$ converges to zero a.s. Consequently the set of limit points of $(2 \log n)^{-1/2}S_n\xi = (2\log n)^{-1/2}PS_n\xi+(2\log n)^{-1/2}QS_n\xi$ must also be contained in $M^{\perp}$, and hence contain no nonzero points of $M$, proving the claim.
\end{proof}

This already proves Proposition \ref{thm1}, and to prove Proposition \ref{thm 2} one only needs the following.

\begin{lem}\label{glue}
Let $(\B,H,\mu)$ be an abstract Wiener space. Suppose that $S_t:H\to H$ is some family of operators satisfying $S_tS_t^* = I$. Moreover assume that $(S_t)$ extends to a strongly continuous semigroup on $\B$. Then there exists a deterministic function $C:[0,1] \to \mathbb R_+$ such that $C(\rho)\to 0$ as $\rho \downarrow 0$ and such that $$\mu\bigg(\bigg\{x \in \B: \limsup_{k \to \infty} \frac{\sup_{t\in[k\rho,(k+1)\rho]} \|S_tx-S_{k\rho}x\|_\B}{\sqrt{\log k}} \leq C(\rho)\bigg\}\bigg) =1$$ for all $\rho \in [0,1]$. In particular $(S_tx)_{t \ge 0}$ has the same set of cluster points as $(S_Nx)_{N \in \mathbb N}$ for $x$ in a set of full $\mu$-measure.
\end{lem}

\begin{proof}
Let $\xi$ denote a random variable in $\B$ with law $\mu$, defined on some probability space $(\Omega, \mathcal F, \mathbb P)$. Then $(S_t\xi)_{t \in [0,1]}$ is continuous in $t$, thus it can be viewed as a Gaussian random variable taking values in the Banach space $C([0,1],\B)$ of continuous paths in $\B$, endowed with norm $\|F\|_{C([0,1],\B)}:=\sup_{t\in[0,1]} \|F(t)\|_{\B}.$ By Fernique's theorem, $\mathbb E[\sup_{t\in [0,1]} \|S_t\xi-\xi\|_{\B}^p]< \infty$ for all $p \ge 1$. By continuity, we also know that $\sup_{t\in [0,\rho]} \|S_t\xi-\xi\|_\B\to 0$ as $\rho \downarrow 0$. Letting
$$C(\rho):=\mathbb E[\sup_{t \in[0,\rho]} \|S_t\xi-\xi\|_\B^2],$$
we thus have by uniform integrability that $\lim_{\rho \downarrow 0} C(\rho) = 0.$ Letting $$A(k,\rho):= \sup_{t\in[k\rho,(k+1)\rho]} \|S_tx-S_{k\rho}x\|_\B,$$
the stationarity of the process $(S_t\xi)_{t \ge 0}$ implies that $A(k,\rho)\stackrel{d}{=}A(0,\rho)$ for all $k \in \mathbb N$. We thus find that $$\mathbb P( A(k,\rho)>\mathbb E[A(0,\rho)]+u) =\mathbb P( A(0,\rho)>\mathbb E[A(0,\rho)]+u) \leq e^{-u^2/(2C(\rho))}, $$where we used Lemma \ref{ledoux} in the last inequality. But $ \mathbb E[A(0,\rho)] \leq \mathbb E[A(0,\rho)^2]^{1/2} =C(\rho)^{1/2}$, and therefore $$\mathbb P(A(k,\rho)>u) \leq  e^{-(u-C(\rho)^{1/2})^2/(2C(\rho))} \leq e^{1/2} \cdot e^{-u^2/(4C(\rho))}$$ where we used $(a-b)^2 \ge \frac12 a^2-b^2$ in the last inequality. Thus $$\mathbb P( A(k,\rho)> 4 C(\rho) \sqrt{\log k})  \lesssim k^{-4},$$ so by the Borel-Cantelli lemma we find that $$\limsup_{k \to \infty} \frac{A(k,\rho)}{\sqrt{\log k}} \leq 4C(\rho),\;\;\;\;a.s.$$ As we already observed eariler $C(\rho) \to 0$ as $\rho \downarrow 0$, completing the proof.
\end{proof}

To summarize the main results of this section, let $(\B,\H,\mu)$ be an abstract Wiener space, let $S:\B\to \B$ be measure-preserving and a.e. linear, and consider the following statements:
\begin{enumerate}
    \item $\bigcap_n \sigma(S_n)$ is a trivial sigma algebra.
    
    \item $S$ is strongly mixing.
    
    \item The set of limit points of the random sequence $((2\log n)^{-1/2}S^nx)_{n\in\mathbb N}$ equals the closed unit ball of $H$ for a.e. $x \in X$.
    
    \item $S$ is ergodic/weakly mixing.
\end{enumerate}
Appendix A shows that (1) implies (2) (which is actually true for any dynamical system, by e.g. the reverse martingale convergence theorem), that (2) implies (3), and that (3) implies (4). From Proposition \ref{prop1}, it is clear that (2) does not imply (1) in general, and that (4) does not imply (2) in general. 

We do not know if (3) implies (2) or if (4) implies (3), although Proposition \ref{prop1} implies that both cannot simultaneously be true in general. To show that (4) implies (3), one would need to prove Lemma \ref{lem:lim1} with the condition $\mathrm{cov}(X_0,X_n)\to 0$ replaced by the condition $\frac1n \sum_{j=1}^n |\mathrm{cov}(X_0,X_j)|\to 0.$ We do not know how to prove this, nor are we certain that it is even true. 
We have reason to suspect that ``(4) implies (3)" may actually be false, and a counterexample might be given by an operator whose spectral measure is atomless but highly singular with respect to Lebesgue measure. For instance, in the case that the spectral measure looks like the usual two-thirds Cantor measure, we have some reason to suspect that the set of limit points of the random sequence $((2\log n)^{-1/2}S^nx)_{n\in\mathbb N}$ equals the ball in $\H$ of radius $\sqrt{\log_3 2}$ for a.e. $x \in X$, rather than the closed unit ball. However, we do not have a proof of this.

\subsection{Strassen law for higher Gaussian chaos}

The main goal of this subsection is to prove a version of Strassen's law for higher Gaussian chaoses, which reduces to the previous version for first-order chaos.

If $(\B,\H,\mu)$ is an abstract Wiener space, then the $k^{th}$ homogeneous Wiener chaos, denoted by $\mathcal H^k(\B,\mu)$ is defined to be the closure in $L^2(\B,\mu)$ of the linear span of $H_k \circ g$ as $g$ varies through all elements of the continuous dual space $X^*,$ where $H_k$ denotes the $k^{th}$ Hermite polynomial
\begin{equation*}
H_k (x) := (-1)^k e^{\frac{x^2}{2}} \frac{d^k}{dx^k} e^{-\frac{x^2}{2}}.  
\end{equation*}
Equivalently $\mathcal H^k$ can be described as the closure in $L^2(\B,\mu)$ of the closure of the linear span of $H_k(\langle \cdot, v\rangle)$ as $v$ ranges through all elements of $\H$ and satisfies $\|v\|_H = 1$. 
One always has the orthonormal decomposition $L^2(\B,\mu) = \bigoplus_{k\ge 0} \mathcal H^k(\B,\mu),$ see e.g. \cite[Section 1.1]{Nua96}. Sometimes the $k^{th}$ chaos is also described slightly differently as linear combinations of \textit{products} of Hermite polynomials of degree adding up to $k$, for instance in \cite[Section 3]{HW13}, but our formulation is equivalent by the umbral identity for Hermite polynomials (see e.g. \cite[Corollary 2.3]{Hai16}) together with the fact that for any vector space $H$ the ``diagonal'' elements of the form $h^{\otimes_s k}\;\;\; (h\in H)$ span span the entire symmetric tensor product $H^{\otimes_s k}$. 

\begin{defn}
Let $(\B,\H,\mu)$ be an abstract Wiener space, and let $Y$ be another separable Banach space. A Borel-measurable map $T:\B\to Y$ is called a homogeneous chaos of order $k$ if $f\circ T \in \mathcal H^k(\B,\mu)$ for all $f \in Y^*$.
\end{defn}

Before formulating the main result of this section, we collect a few important results about homogeneous variables. In \cite{AG93} the authors show that the norm of any Gaussian chaos has moments of all orders. Moreover, \cite[equation (4.1)]{AG93} gives a hypercontractive bound for any Gaussian chaos $T$ of order $k$: 
\begin{equation}\label{hyp}\int_X \|T(x)\|_Y^{2p} \mu(dx) \leq (2p-1)^{pk} \bigg[\int_X \|T(x)\|_Y^{2} \mu(dx)\bigg]^p. 
\end{equation}
Here $p\ge 1$ is arbitrary. Immediately this gives the following strong integrability result.

\begin{prop}\label{eq'}Let $(\B,\H,\mu)$ be an abstract Wiener space and let $T:\B\to Y$ be a chaos of order $k.$ Let $\|T\|_{L^2(\B,\mu;Y)}^2:=\int_\B\|T(x)\|_Y^2\mu(dx).$ Then $\|T\|_{L^2(\B,\mu;Y)}<\infty$ and moreover
\begin{equation}\label{b'}\mu(\{x \in \B: \| T(x)\|_Y >u\}) < C \exp\big[-\alpha \big(u/ \|T\|_{L^2(\B,\mu;Y)}\big)^{2/k}\big],
\end{equation}
where $C,\alpha>0$ depend on $k$ but are independent of the choice of $\B,\H,\mu,Y,T,$ and $u >0$.
\end{prop}


Let $(\B,\H,\mu)$ be an abstract Wiener space. Choose an orthonormal basis $\{e_i\}_i$ for $\H,$ and for $x\in H$ let $$P_Nx:= \sum_{j=1}^N \langle x,e_j\rangle e_j,\;\;\;\;\;\;\;\;\;\;\;\;\;\; Q_Nx := \sum_{j=N+1}^{\infty} \langle x,e_j\rangle e_j=x-P_Nx.$$
where the infinite sum converges in the topology of $H$. Note that if $x$ is sampled from $\mu$ then $P_Nx,Q_Nx$ still make sense (in a set of measure 1) and moreover they are independent since $\{\langle x,e_j\rangle\}_{j\ge 1}$ are i.i.d. standard Gaussian random variables under $\mu$. Therefore if $(x,y)$ is sampled from $\mu^{\otimes 2}$, then $P_Nx+Q_Ny$ is distributed as $\mu$. If $T:\B\to Y$ is a chaos of order $k$, we thus define a sequence of ``finite-rank Cameron-Martin projections" for $T$ by the formula
\begin{equation}\label{frcmp'}
    T_N(x):= \int_\B T(P_Nx+Q_Ny) \mu(dy).
\end{equation}
This is a well-defined Bochner integral for $\mu$ a.e. $x\in\B$. Indeed, since 

\begin{equation*}
\int_\B \int_\B \|T(P_Nx+Q_Ny)\|_Y \mu(dy)\mu(dx)=\int_\B \|T(u)\|_Y\mu(du)<\infty,
\end{equation*} 
it follows that $\int_\B \|T(P_Nx+Q_Ny)\|_Y \mu(dy)$ is finite for  a.e. $x$.

\begin{cor}\label{approximation lemma'}
Let $(\B,\H,\mu)$ be an abstract Wiener space, and let $T:\B\to Y$ be a chaos of order $k$. If $T_N$ is defined as in \eqref{frcmp'}, then $T_N$ is also a chaos of order $k$, and moreover $\|T_N- T\|_Y\to 0$ a.e. and in every $L^p(\B,\mu)$ as $N\to \infty$. In fact, one has the following super-polynomial convergence bound:
\begin{equation}\label{mis'}
   \mu(\{x\in\B: \|T_N(x)-T(x)\|_Y>u\})\leq C\exp\big[-\alpha \big(u/\|T_N-T\|_{L^2(\B,\mu;Y)}^2\big)^{2/k} \big]
\end{equation}
where $C,\alpha$ are independent of $\B,\H,\mu,Y,T,u,N$ and the choice of basis $\{e_i\}_i$, but may depend on the homogeneity $k$.
\end{cor}

\begin{proof}That $T_N$ is a chaos of order $k$ is clear. The a.e. and $L^2$ convergence follows immediately from the martingale convergence theorem for Banach-valued random variables \cite{ttttttttttttt}, since $T_N = \mathbb E_\mu[T|\mathcal F_N]$ where $\mathcal F_N$ is the $\sigma$-algebra generated by the i.i.d. variables $\langle x,e_i\rangle$ for $1\le i \le N$. The super-polynomial convergence bound then follows immediately from Proposition \ref{eq'} applied to the chaos $T-T_N.$ 
\end{proof}

If $(\B,\H,\mu)$ is an abstract Wiener space and $T:\B\to Y$ is a chaos of order $k$, then we define the associated map $T_{\mathrm{hom}}:H\to Y$ by \begin{equation}\label{homm}T_{\mathrm{hom}}(h):= \int_\B T(x+h)\mu(dx) = \frac1{k!}\int_\B T(x)\langle x,h\rangle^k\mu(dx).
\end{equation}
The latter equality follows by applying the Cameron-Martin formula to first replace $T(x+h)$ by $T(x)e^{\langle x,h\rangle -\frac12\|h\|_H^2}$, then expanding the exponential as $\sum_{n = 0}^\infty \frac{\|h\|^n}{n!} H_n(\langle x,h/\|h\|\rangle)$ and using the fact that $T$ is orthogonal to all $H_n \circ g$ when $n \ne k$ and $g\in X^*$. 
\begin{lem}\label{hom is cont'}
Let $(\B,\H,\mu)$ be an abstract Wiener space, and let $T:\B\to Y$ be a chaos of order $k$. Choose an orthonormal basis $\{e_i\}$ of $\H$ and let $T_N$ be the finite-rank approximation given in \eqref{frcmp'}. Then we have the uniform convergence $$\lim_{N\to \infty} \sup_{\|h\|_H\le 1} \big\|(T_N)_{hom}(h)-T_{\mathrm{hom}}(h)\big\|_Y=0.$$ 
Letting $B(H)$ denote the closed unit ball of $\H$, it follows that $T_{\mathrm{hom}}$ is continuous from $B(H)\to Y$, where $B(H)$ is given the topology of $\B$ (not of $\H$).
\end{lem}
This is proved in \cite[(3.8)]{HW13}. Since continuous functions map compact sets to compact sets, we immediately obtain the following.
\begin{cor}
Let $(\B,\H,\mu)$ be an abstract Wiener space, and let $T^i:\B\to Y_i$ be a chaos of order $k_i$ for $1\le i \le m$, where $m\in \mathbb N$. Then the set $$\{(T^1_{\mathrm{hom}}(h),...,T^m_{\mathrm{hom}}(h)):h\in B(H)\}$$ is a compact subset of $ Y_1\times \cdots Y_m$. 
\end{cor}
With all of these preliminaries in place, we are ready to formulate the main theorem of this subsection, which is a generalization of the main theorem to homogeneous variables of order $k$.

\begin{thm}\label{thm 4}
Let $(\B,\H,\mu)$ be an abstract Wiener space, fix $m\in \mathbb N$. Let $T^i:\B\to Y_i$ be chaoses of order $k_i$ respectively, for $1\le i \le m$. Let $S:\B\to \B$ be a.e. linear, measure-preserving, and strongly mixing. Then almost surely, the set of limit points of the random set 
$$\big\{ \big( (2\log n)^{-k_1/2} T^1(S_nx),\;.\;.\;.\;,(2\log n)^{-k_m/2}T^m(S_nx)\big): n \in \mathbb N\}$$ is equal to the compact set $K:=\{(T^1_{\mathrm{hom}}(h),...,T^m_{\mathrm{hom}}(h)):h\in B(H)\}. $
\end{thm}
\begin{proof}
Recall that $S(H)=\{h\in\H:\|h\|_H=1\}.$ Note by Lemma \ref{lem:s=b} that $S(H)$ is dense in $B(H)$ with respect to the topology of $\B$. Therefore, the set $$D:=\{(T^1_{\mathrm{hom}}(h),...,T^m_{\mathrm{hom}}(h)):h\in S(H)\}$$ is dense in $K$. Thus it suffices to show that any point in $D$ is a limit point of the given sequence. So fix a point $(T^1_{\mathrm{hom}}(h),...,T^m_{\mathrm{hom}}(h))\in D$, where $\|h\|_H=1$, and let $\epsilon>0$. We wish to show that 
\begin{equation}\label{blue}\sum_{i=1}^m \big\| (2\log n)^{-k_i/2} T^i(S_nx) - T^i_{hom}(h)\big\|_{Y_i}<\epsilon
\end{equation}
infinitely often. Fix an orthonormal basis $\{e_i\}$ for $\H$, with $e_1=h$, and let $T^i_N$ be the associated finite rank Cameron-Martin projections as in \eqref{frcmp'}. We claim that it is enough to prove \eqref{blue} with each $T^i(S_nx)$ replaced by $T^i_N(S_nx)$ and $T^i_{hom}(h)$ replaced by $(T_N^i)_{hom}(h)$ (and also replacing $\epsilon$ by $\epsilon/2$), for some large enough $N$. 

Indeed, by Corollary \ref{approximation lemma'} we can choose $N$ so large that $\int_X \|T^i_N-T^i\|_{Y_i}^2d\mu<\epsilon/(2m).$ Then by \eqref{mis'} and the fact that $S_n$ is measure preserving, it is clear that 
\begin{equation}\label{forn}\sum_{n\ge 2} \mu(\{x: (2\log n)^{-k_i/2} \|T^i_N(S_nx)-T^i(S_nx)\|_Y>\epsilon/m\}) <\infty,
\end{equation}so by Borel-Cantelli lemma, $(2\log n)^{-k_i/2} \|T^i_N(S_nx)-T^i(S_nx)\|_Y<\epsilon/m$ for all but finitely many $n\in \mathbb N$ almost surely. Furthermore, by Lemma \ref{hom is cont'} we can (by making $N$ larger) ensure that $\|(T^i_N)_{hom}(h) - T^i_{hom}(h)\|_Y<\epsilon/m.$

Thus we just need to show that \begin{equation}\label{blue'}\sum_{i=1}^m \big\| (2\log n)^{-k_i/2} T^i_N(S_nx) - (T^i_N)_{hom}(h)\big\|_{Y_i}<\epsilon/2
\end{equation}
infinitely often. Note that each $T^i$ is a chaos of order $k$ and measurable with respect to a \textit{finite} collection $\{\langle \cdot,e_i\rangle\}_{i=1}^N$ of i.i.d. standard Gaussian random variables, therefore one can verify that
 each $T^i_N$ can be written as \begin{equation}\label{ct1}T^i_N(x) = \sum_{j=1}^{M_i} y_j^i H_{k_i}(\langle v^i_j,x\rangle), \end{equation}
for some \textit{finite} collection of vectors $\{v^i_j\}\subset \mathrm{span}(\{e_i\}_{i=1}^N)$, $\{y^i_j\}\subset Y$, and $M_i\in\mathbb N$. 
Here $H_k$ is the $k^{th}$ Hermite polynomial.
Using Cameron-Martin theorem, one can see that 
\begin{equation}\label{ct2}(T^i_N)_{hom}(h) = \sum_{j=1}^{M_i} y_j^i \langle v^i_j,h\rangle^{k_i}.
\end{equation}
Note that $$\big|\langle v^i_j , (2\log n)^{-1/2}S_nx\rangle^{k_i} - (2\log n)^{-k_i/2}H_{k_i}(\langle v^i_j,S_nx\rangle)\big|\to 0\;\;\;\;\;\;\;a.s.,$$ by 
Borel-Cantelli lemma and the fact that $x\mapsto \langle v^i_j,S_nx\rangle$ are standard Gaussian random variables under $\mu$. Thus it suffices to show that
$$\sum_{\substack{1\le i \le m \\ 1\le j \le M_i}}\|y^i_j\|_Y\big|\langle v^i_j, (2\log n)^{-1/2}S_nx \rangle^{k_i} - \langle v^i_j,h\rangle^{k_i}\big|<\epsilon/2 $$
happens infinitely often. Letting $P_N:\H \to \mathrm{span}\{e_i\}_{i=1}^N$ denote the orthogonal projection, it is clear that $v^i_j = P_N(v^i_j)$, thus by exploiting self-adjointness of $P_N$, the previous expression is equivalent to showing that $$\sum_{\substack{1\le i \le m \\ 1\le j \le M_i}}\|y^i_j\|_Y\big|\langle v^i_j, (2\log n)^{-1/2}P_N(S_nx) \rangle^{k_i} - \langle v^i_j,h\rangle^{k_i}\big|<\epsilon/2 $$ infinitely often. Since 
$h = e_1$, we know from Corollary \ref{cor:sphere} that $\|(\log n)^{-1/2}P_N(S_nx)-h\|_H<\delta$ infinitely often, for arbitrary $\delta>0$. By choosing $\delta$ small enough and noting that $\langle v^i_j,\cdot\rangle$ is a continuous function on $\text{span}\{e_i\}_{i=1}^N$, the claim \eqref{blue'} immediately follows, and thus \eqref{blue} is proved.

Now the only thing left to show is that the set of limit points of the given sequence cannot contain points outside of the set $K.$ This will be done in the following lemma.
\end{proof}

\begin{lem} \label{ar} In the setting of Theorem \ref{thm 4}, let $$a_n(x):=\big((2\log n)^{-k_1/2} T^1(S_nx),\;.\;.\;.\;,(2\log n)^{-k_m/2}T^m(S_nx)\big).$$ For a.e $x\in X$ one has that $\mathrm{dist}(a_n(x),K) <\epsilon$ for all but finitely many $n$.

\end{lem}

\begin{proof} Note that if our abstract Wiener space $(\B,\H,\mu)$ is finite dimensional, then the statement is straightforward, since $T$ and $T_{\mathrm{hom}}$ are of the form \eqref{ct1} and \eqref{ct2} respectively, and since all of the relevant quantities are continuous functions. 

Now we move to the infinite-dimensional case. Suppose for contradiction the claim was false. Let $U$ denote a neighborhood of size $\epsilon$ around $K$. Then since $S$ is strongly mixing (hence ergodic) and since the event ``$ \mathrm{dist}( U^c, a_n(x))<\epsilon/5$ infinitely often" is shift invariant, it follows that it actually occurs with probability 1. By the same argument used in deriving \eqref{forn}, we can choose $N$ so large that $\sum_{n=1}^{m} (2\log n)^{-k_i/2} \|T^i(S_nx)-T^i_N(S_nx)\|_{Y_i} <\epsilon/5$ for all but finitely many $n$ almost surely, and moreover by Lemma \ref{hom is cont'} we can ensure (by making $N$ possibly larger) that $\sum_{n = 1}^m \sup_{\|h\|\le 1} \|T^i_{hom}(h) - (T^i_N)_{hom}(h)\|_{Y_i} <\epsilon/5.$ By the latter bound and the definition of $U$ it is clear that $$dist\bigg(\; U^c\;,\; \big((T^1_{\mathrm{hom}})_N(h),...,(T^m_N)_{hom}(h)\big)\;\bigg) >4\epsilon/5$$ for any $h$ such that $\|h\|_H\le 1$. On the other hand the former bound and our shift-invariant event of full probability guarantees that $\mathrm{dist}( U^c,a^N_n(x))<2\epsilon/5$ infinitely often (for a.e. $x$), where $$a_n^N(x):=\big((2\log n)^{-k_1/2}T^1_N(S_n x),...,(2\log n)^{-k_m/2} T^m_N(S_nx)\big).$$ The preceding two sentences imply (by finite dimensionality) that the sequence $(a_n^N(x))_{n\ge 1}$ contains a limit point outside of the set $$\{ ((T^1_N)_{hom}(h) ,..., (T^m_N)_{hom}(h)) : \|h\|_H \leq 1\}, $$ since the distance of $a_n^N(x)$ to that set must be greater than $2\epsilon/5$ infinitely often. This contradicts the finite dimensional version of the statement that the set of limit points must be contained in $K$, which is impossible as noted earlier.
\end{proof}

Next we formulate a continuous-time version of the above results. If $Y$ is a Banach space, we denote by $C([0,1],Y)$ the space of continuous maps from $[0,1]\to Y,$ equipped with the Banach space norm $\|\gamma\|_{C([0,1],Y)}:=\sup_{t\in[0,1]}\|\gamma(t)\|_Y.$ For $t\in [0,1]$ we define $\pi_t: C([0,1],Y)\to Y$ by sending $\gamma \mapsto \gamma(t).$

\begin{thm}\label{ar2}
Let $(\B,\H,\mu)$ be an abstract Wiener space, let $(S_t)_{t\ge 0}$ be a family of Borel-measurable a.e. linear maps from $\B\to \B$ which are measure-preserving and strongly mixing, and let $T^i:\B\to Y_i$ be homogeneous of degree $k_i$ for $1\le i \le m$. Suppose that there exist strongly continuous semigroups $(G^i_t)_{t\ge 0}$ of operators from $Y_i\to Y_i$ for $1\le i \le m$ with the property that 
\begin{equation}\label{commute_a}T^i \circ S_t = G^i_t \circ T^i, \;\;\;\;\;\;\;\; \mu\text{-}\mathrm{a.e.}\;\;\;\;\;\; \text{ for all } t \ge 0.
\end{equation}
Then almost surely, the set of cluster points at infinity of the random set 
$$\big\{ \big( (2\log t)^{-k_1/2} G_t^1T^1(x),\;.\;.\;.\;,(2\log t)^{-k_m/2}G_t^mT^m(x)\big): t \in  [0,\infty)\}$$ is equal to the compact set $K:=\{(T^1_{\mathrm{hom}}(h),...,T^m_{\mathrm{hom}}(h)):h\in B(H)\}. $
\end{thm}

Note that we impose no semigroup condition on $(S_t)$ itself. This is because we do not need to, though in practice $(S_t)$ will usually be a strongly continuous semigroup on $\B$. 
Note that this theorem implies Proposition \ref{thm 2}: set $k=1$, $m=1$, $Y=\B$, let $T^1$ be the identity on $\B$, and let $G_t^1=S_t$.

\begin{proof}
First we claim that that $(G_t^i(T^i(\xi)))_{t\in [0,1]}$ is homogeneous variable of order $k_i$ taking values in the space $C([0,1],Y_i).$ To prove this, note that if $Y,Z$ are Banach spaces, if $T:\B\to Y$ is a chaos of order $k$, and if $A:Y\to Z$ is a bounded linear map, then $A\circ T$ is also a chaos of order $k$. We simply apply this to the case where $Y=Y_i$, $Z = C([0,1],Y_i)$ and $A:Y_i \to C([0,1],Y_i)$ sends a point $y$ to $(G^i_ty)_{t\in [0,1]}$. This linear map is bounded by the uniform boundedness principle. 

Next, note that the set of cluster points must contain $K$ by Theorem \ref{thm 4}. Thus we just need to show it contains no other points. Consider the random set $$\big\{ \big( (2\log t)^{-k_1/2} G_t^1(T^1(x)),\;.\;.\;.\;,(2\log t)^{-k_m/2}G_t^m(T^m(x))\big): t \in  [0,\infty)\}.$$
The argument that this set contains no cluster points outside of $K$ is very similar to that of Lemma \ref{glue}. More precisely, we show that there exists deterministic functions $C_i:[0,1] \to \mathbb R_+$ such that $C_i(\rho)\to 0$ as $\rho \downarrow 0$ and such that $$\mu\bigg(\bigg\{x \in \B: \limsup_{n \to \infty} \frac{\sup_{t\in[n\rho,(n+1)\rho]} \|G_t^i(T^i(x))-G_{n\rho}^i(T^i(x))\|_{Y_i}}{(\log n)^{k_i/2}} \leq C_i(\rho)\bigg\}\bigg) =1$$ for all $\rho \in [0,1]$ and $1\le i \le m$. In particular $(G_t^i(T^i(x)))_{t \ge 0}$ has the same set of cluster points as $(G_N(T^i(x)))_{N \in \mathbb N}$ for $x$ in a set of full $\mu$-measure.

To prove the above claim, one uses precisely the same arguments as we did in the proof of Lemma \ref{glue}. Namely one defines $C_i(\rho):= C\mathbb E[\sup_{t\in [0,\rho]} \|G_t^i(T^i(\xi))-T^i(\xi)\|_{Y_i}^2]$, where $\xi$ is sampled from $\mu$ and $C>0$ is to be determined later. Then one uses the fact that $(G_t^i(T^i(\xi)))_{t\in [0,1]}$ is homogeneous variable of order $k_i$ taking values in the space $C([0,1],Y_i),$ by the discussion above. Finally one uses the associated tail bounds for such homogeneous variables as given in Proposition \ref{eq'}, and concludes using the Borel-Cantelli lemma.
\end{proof}

\subsection{Contraction principle for Strassen Laws}

We derive a corollary that will be used in deriving Strassen's Law for singular semilinear SPDEs later. The following can be viewed as a sort of contraction principle for Strassen's law under continuous maps which may be nonlinear.

\begin{cor}\label{contraction}
Let $(\B,\H,\mu)$ be an abstract Wiener space, and let $T^i:\B\to Y_i$ be a chaos of order $k_i$ for $1\le i \le m$, where $m\in \mathbb N$.  Let $Z$ be a Banach space, and let $\mathcal M\subset Y_1\times \cdots Y_m$ be a closed subset, such that for all $\delta>0$ one has 
\begin{equation}\label{eq:assumption}
\mu(\{x\in X: (\delta^{k_1} T^1(x),...,\delta^{k_m} T^m(x)) \in \mathcal M\})=1.
\end{equation}
Let $\Phi:\mathcal M \to Z$ be continuous. Then the compact set $K:=\{(T^1_{\mathrm{hom}}(h),...,T^m_{\mathrm{hom}}(h)):h \in B(H)\}$ is necessarily contained in $\mathcal M$, and moreover the set of cluster points at infinity of the random set 
$$\big\{ \Phi\big( (2\log n)^{-k_1/2} T^1(S_nx),\;.\;.\;.\;,(2\log n)^{-k_m/2}T^m(S_nx)\big): n\in \mathbb N\big\}$$ is almost surely equal to $\Phi(K).$
\end{cor}

\begin{proof}
Note that if $(a_1,...,a_m) \in K$ then by Theorem \ref{thm 4} there exists a subset $E\subset X$ of full measure  such that for each $x\in E$, the sequence $\big((2\log n)^{-k_1/2} T^1(S_nx),\;.\;.\;.\;,(2\log n)^{-k_m/2}T^m(S_nx)\big)$ converges to $(a_1,...,a_m)$ along some subsequence as $n \to \infty$. Since $S_n$ are measure-preserving it holds by 
\eqref{eq:assumption} that $$\big((2\log n)^{-k_1/2} T^1(S_nx),\;.\;.\;.\;,(2\log n)^{-k_m/2}T^m(S_nx) \big)\in \mathcal M$$ for a.e. $x\in X$. Consequently $(a_1,...,a_m)$ is a limit point of the closed set $\mathcal M$ and thus 
belongs to $\mathcal M$. This implies that $K$ is contained in $\mathcal{M}$.

Now we prove that the limit set is necessarily $\Phi(K)$. The fact that any point $z\in \Phi(K)$ is a limit point is 
due to the fact  that $\Phi$ is continuous and any point of $K$ is a limit point of the sequence $$\big((2\log n)^{-k_1/2} T^1(S_nx),\;.\;.\;.\;,(2\log n)^{-k_m/2}T^m(S_nx)\big)$$ 
as $n \to \infty$ (see Theorem \ref{thm 4}).

Now we need to prove that points outside of $\Phi(K)$ are not limit points. Suppose $z\notin \Phi(K)$. The latter set is closed so we may choose $\epsilon >0$ such that $\|z-b\|_Z >\epsilon$ for all $b\in \Phi(K)$. Choose $\delta>0$ so that $\mathrm{dist}(\Phi(a), \Phi(K)) < \epsilon$ whenever $\mathrm{dist}(a,K) <\delta$ (this $\delta$ exists by compactness of $K$). 
We choose points $a_1,...,a_N$ so that $B(a_i,\epsilon)$ form an open cover of $\Phi(K)$, then consider the open cover $U_i:=\Phi^{-1}(B(a_i,\epsilon))$ of $K$, then let $U$ denote the union of the $U_i$, and take $\delta:=\min_{x\in K} \mathrm{dist}(x,U^c)>0$). Letting $a_n(x):=\big((2\log n)^{-k_1/2} T^1(S_nx),\;.\;.\;.\;,(2\log n)^{-k_m/2}T^m(S_nx)\big)$, by Lemma \ref{ar} we know that for a.e $x\in X$ that $\mathrm{dist}(a_n(x),K) <\delta$ for all but finitely many $n$ and therefore $z$ is not a limit point of $\Phi(a_n(x))$.
\end{proof}

\begin{cor}\label{contraction 2}
Let $(\B,\H,\mu)$ be an abstract Wiener space, let $(S_t)_{t\ge 0}$ be a family of Borel-measurable a.e. linear maps from $\B\to \B$ which are measure-preserving and strongly mixing, and let $T^i:\B\to Y_i$ be homogeneous of degree $k_i$ for $1\le i \le m$. Suppose that there exist strongly continuous semigroups $(G^i_t)_{t\ge 0}$ of operators from $Y_i\to Y_i$ for $1\le i \le m$ with the property that 
\begin{equation*}T^i\circ S_t = G^i_t \circ T^i, \;\;\;\;\;\;\;\; \mu\text{-a.e.}\;\;\;\;\;\; \text{ for all } t \ge 0.
\end{equation*}
Let $Z$ be a Banach space, and let $\mathcal M\subset Y_1\times \cdots Y_m$ be a closed subset, such that for all $\delta>0$  $$\mu(\{x\in X: (\delta^{k_1} T^1(x),...,\delta^{k_m} T^m(x)) \in \mathcal M\})=1.$$
Let $\Phi:\mathcal M \to Z$ be continuous. Then the compact set $K:=\{(T^1_{\mathrm{hom}}(h),...,T^m_{\mathrm{hom}}(h)):h \in B(H)\}$ is necessarily contained in $\mathcal M$, and moreover the set of cluster points at infinity of
$$\big\{ \Phi\big( (2\log t)^{-k_1/2} G_t^1T^1(x),\;.\;.\;.\;,(2\log t)^{-k_m/2}G_t^mT^m(x)\big): t \in [0,\infty)\}$$ is almost surely equal to $\Phi(K).$
\end{cor}

\begin{proof}
    The argument is similar to the previous corollary but applying Theorem \ref{ar2} rather than Theorem \ref{thm 4}.
\end{proof}

Note that the previous two corollaries are the most general version of the Strassen's Law that we have stated so far (e.g. set $Z=Y_1\times \cdots Y_m =\mathcal M$ and let $\Phi$ be the identity). Note also that Corollary \ref{contraction 2} will be used in multiplicative form later (see Theorem \ref{mainr}), not the additive form stated above. Even further generalizations will be given in Theorems \ref{contraction 3}, \ref{gen}, and \ref{repar} below.


\section{Examples and applications to SPDEs}\label{rats}

With the above results in place, we now move on to the main result of the paper. Although we want to study the application of the above results to singular SPDEs such as KPZ, we start with a series of simpler examples just to illustrate the methods.

\subsection{First examples} The first three examples will be done with just the first chaos (in other words just a Gaussian measure) so that only Proposition \ref{thm 2} is needed as opposed to the full generality of Theorem \ref{mainr}. Then the next two examples will be done with higher chaos and contraction by continuous maps. A few of these examples will concisely recover classical results.
\begin{example}[Brownian and fractional Brownian motion]\label{41}

Consider the fractional Brownian motion $(B^H(t))_{t \in [0, 1]}$ with Hurst parameter $H \in (0, 1)$. Let $\mathcal F$ denote the space of smooth real-valued functions on $[0,1]$ which vanish at $0$. We can let $X$ be the closure of $\mathcal F$ in the Hölder space $C^\kappa[0,1]$ for $\kappa \in (0,H)$. In fact, one may consider the closure of $\mathcal F$ with respect to an even stronger norm $$\|f\|_X = \sup_{0<|t-s|<1/3} \frac{|f(x)-f(y)|}{\varphi(|t-s|)}$$ where $\varphi$ may be any increasing function vanishing at $0$ whose growth near $0$ is strictly larger than the modulus of continuity of $B^H$ (for instance $\varphi(a) = a^{H} (\log(1/a))^{\frac12+\gamma}$ and $\gamma>0$ is small). One may apply Theorems 2.52 and 8.27 in \cite{lipalg} to see that the law of $B^H$ is supported on the closure of smooth functions in such a norm (in fact, the closure contains any function with a \textit{strictly} better modulus of continuity than $\varphi$).

One defines the measure-preserving operator $R_\e f(t):= \e^{-H} f(\e t)$. Since $\varphi$ is increasing we see that $\sup_{\e\in[e^{-1},1]}\sup_{a\in(0,1/4]} \e^{-H}\varphi(\e a)/\varphi(a)\leq e^{H}<\infty$, therefore Lemma \ref{sg} implies that $(R_\epsilon)$ is strongly continuous on this Banach space. Now to prove the Strassen law, we need to verify a strong mixing condition on the appropriate space. Since this mixing condition is the crucial input needed to yield the Strassen law, and since one of the main contributions of this work is to reduce the proof of the Strassen law to this deterministic mixing calculation, we give several different proofs of the mixing condition.

\begin{enumerate}

\item A result of \cite[(6)]{DU99} shows that the Cameron-Martin space $\mathcal H$ of $B^H$ consists of functions of the form $h(t) = \int_0^t K_H(t, s) \ell(s) ds$ for some $\ell\in L^2[0,1]$, with the isometric identity $\|h\|_\mathcal H = \|\ell\|_{L^2}$, where $K_H$ is an explicit kernel that satisfies $K_H (\epsilon t, \epsilon s) = \epsilon^{H-\frac{1}{2}} K_H (t, s)$. Consequently one has $\|R_\e h\|_\mathcal H = \|\tilde R_\e \ell\|_{L^2[0,1]}\to 0$ as $\e\to 0$, where $\tilde R_\e \ell (t) = \e^{1/2} \ell(\e t).$

\item Again using \cite[(6)]{DU99}, one may write $B^H(t) = \int_0^t K_H(s,t) dB^{\frac12}(s)$ for a standard Brownian motion $B^{\frac12}$ on $[0,1]$. If we let $\tilde X:=C[0,1]$, $\tilde{\mathcal H}=H_0^1[0,1]$ and $\mu$ the law on $X$ of $C[0,1]$, then we may then view $B^H$ as an $X$-valued chaos of order $1$ over $(\tilde X, \tilde{\mathcal H}, \mu)$. More precisely we have $B^H = \psi(B^{1/2})$ where $\psi(B) = \int_0^\bullet K_H(s,\bullet) dB(s).$ Note that $\psi_{hom}(\ell) = \int_0^\bullet K_H(s,\bullet)\ell(s)ds$ for $\ell \in L^2[0,1].$ Moreover, we clearly have the relation $R_\e \psi = \psi \tilde R_\e$ where $\tilde R_\e$ is the same semigroup on $L^2[0,1]$ as in the proof (1) above. Thus we may conclude the Strassen law from Theorem \ref{mainr}. While this perspective is equivalent to the proof used in (1) just above, the paradigm will be more useful when we deal with higher chaos.

\end{enumerate}

This concludes the proof for the Strassen law of $B^H$ on the unit interval $[0,1]$. Rather than considering the Strassen law for $B^H$ restricted to the unit interval, we can instead consider it on the full real line $\mathbb R$ as well. Indeed, the Banach space $X$ can be taken to be $X := \{f\in C[0,\infty) : \lim_{t\to\infty} \frac{|f(t)|}{t} = 0\}$ with norm given by $\sup_{t\ge 0} \frac{|f(t)|}{1+t}$. This is a separable Banach space which supports $B^H$ (for all time, not just the unit interval) almost surely. Alternatively one can also use Holder spaces with suitable weight functions. 

In this full-line setting, we can give alternative and more self-contained proofs of the calculation of the Cameron-Martin space and the mixing property of the operators $R_\e$ (defined by the same formula) without using \cite{DU99}, and instead using results from Appendix B. Letting $C(s,t):= |t|^{2H} +|s|^{2H} - |t-s|^{2H}$, we can define a bounded operator on $L^2[0,1]$ by $Qf(t) = \frac12\int_0^1 C(s,t) f(s)ds.$ In the sense of tempered distributions the Fourier transform of $a(t)=|t|^{2H}$ is given by (the finite part of) $\hat a(\xi) = c_H |\xi|^{-2H-1}$ for some universal $c_H>0$ (this is immediate from a scaling argument). This implies that if $\phi\in C_c^\infty(\mathbb R)$, then 
\begin{equation}\label{ders}(Q\phi',\phi')_{L^2[0,1]} = -\frac12\int_\mathbb R \int_\mathbb R |t-s|^{2H} \phi'(s)\phi'(t) dsdt = \frac{c_H}2\int_\mathbb R |\hat \phi(\xi)|^2 |\xi|^{1-2H} d\xi,\end{equation} where $\phi'$ is the usual derivative, and the last equality can be deduced using Plancherel's identity and the convolution-to-multiplication property of the Fourier transform. Since $|\xi|^{1-2H}$ has polynomial growth or decay at infinity, this identity is enough to show that the completion of $L^2(\mathbb R)$ with respect to the norm $(Q\phi,\phi)^{1/2}$ is contained in the space $\mathcal S'(\mathbb R)$ of tempered distributions. 
Using Lemma \ref{opdu} we see that the image of $Q^{1/2}$ contains $\mathcal S(\mathbb R)$ and, regardless of the choice of norm on the larger space $X$, the Cameron-Martin space of $B^H$ is given by the completion of $\mathcal S(\mathbb R)$ under $\|\phi\|_\mathcal H = \|Q^{-1/2}\phi\|_{L^2}$. 
Now to verify the mixing condition, note that if $f,g$ are any smooth functions such that $f',g'$ have compact support, then 
\begin{equation}\label{3.2a}\langle R_\epsilon f,g\rangle_H =  (\epsilon\cdot R_\epsilon(f'),Q_{\mathrm{der}}^{-1}g')_{L^2[0,1]} = \frac{2}{c_H}\int_{\mathbb R} |\xi|^{2H-1} \epsilon^{-H-1} \cdot i\xi \hat f(\epsilon^{-1}\xi) \cdot \overline{i\xi \hat g(\xi)} d\xi,
\end{equation}
where $Q_{\mathrm{der}}:= -\frac{d}{dt}Q\frac{d}{dt}$ whose Fourier symbol is $\frac{c_H}{2} |\xi|^{1-2H}$ by \eqref{ders}. Note that by setting $\epsilon=1$, this shows that $\langle f,g\rangle_H = \frac{2}{c_H} \int_\mathbb R |\xi|^{2H+1} \hat f(\xi) \overline{\hat g(\xi)} d\xi.$

Using the fact that $\epsilon^{-1} \hat f(\epsilon^{-1}\xi)$ is an approximate delta function, the right side of \eqref{3.2a} can be bounded above by $C\epsilon^{-1-H} \cdot \epsilon \cdot \epsilon^{2H+1} = C\epsilon^{1+H}$ where $C$ may depend on $f,g$ but not on $\epsilon\in (0,1]$. %
We may now use Lemma \ref{sg'}, noting that 
such $f,g$ are dense in $\mathcal S(\mathbb R)$. 
Then Strassen's law for $(\log\log(1/\e))^{-1/2}R_\e B^H$ thus follows from the multiplicative form of Proposition \ref{thm 2}. This concludes the alternative direct proof of the Strassen's law (on the full line). 

We remark that this full-line proof generalizes to yield Strassen's law for the multiparameter fractional Brownian sheet studied in \cite{CWZ04}, in every Hölder space. Rather than considering the small-time version of Strassen's law for fractional Brownian motion, we can also directly prove the large-time version. 
Indeed one similarly checks that the family of operators $\tilde R_\epsilon f(x) = \epsilon^{H} f(\epsilon^{-1} x) $ is strongly mixing with respect to this norm, using the result of Lemma \ref{sg'}. Consequently one obtains the Strassen's Law for $\epsilon^{H} B(\epsilon^{-1}t)/\sqrt{2\log\log(1/\epsilon)}$ with respect to the norm of $X$.
\end{example}

\begin{example}\label{additive she}

Consider the stochastic heat equation
\begin{equation}\label{she0}
\pa_t \h(t, x)= \pa^2_{x} \h(t, x) + \xi(t, x), \quad \h(0, x) = 0,\quad t \ge 0, x\in \mathbb R^d,
\end{equation}
where $\xi$ is a Gaussian space-time white noise on $\mathbb R_+\times \mathbb R$. The solution is defined by $h=K*\xi$, where $K(t,x) = \frac1{\sqrt{2\pi t}}{e^{-x^2/2t}}1_{\{t\ge 0\}}, $ where $*$ denotes convolution in both space and time, and the convolution may be interpreted as a stochastic integral against $\xi$.

In the notation of Appendix B, let $U=\mathbb (0,\infty)\times \mathbb R^d$ and we define the Frechet space $\mathcal F(U)$ to be the set of all smooth functions of rapid decay on $\mathbb R_+\times \mathbb R^d$ such that $f(0,x) = 0$ for all $x\in \mathbb R$. 
For $d=1$ one may show that the solution $h$ is supported on the space $X$ given by the completion of $\mathcal F(U)$ with respect to $\|\phi\|_X:= \sup_{t,x} \frac{\phi(t,x)}{1+t+|x|}.$ For $d>1$ the equation cannot be realized as a continuous function, thus one has to use a space of generalized functions for $X$. An example would be the closure of Schwartz functions in a weighted parabolic Besov-Hölder space of negative exponent $\alpha:=\frac12-\frac{d}2-\kappa$ where $\kappa>0$ is arbitrary. More specifically, one takes the closure of $\mathcal F(U)$ under
\begin{equation}\label{choice2}\|\phi\|_X:= \sup_{\|\psi\|_{C^{r}\leq 1}} \sup_{\lambda\in(0,1],(t,x) \in \mathbb R_+\times \mathbb R} \frac{\lambda^{-\alpha}|(\phi,\psi^{\lambda}_{t,x})|}{w(t,x)}, 
\end{equation}
where $r = -\lfloor \alpha \rfloor$, $\psi^{\lambda}_{t,x}(s,y) = \lambda^{-d-2} \phi(\lambda^{-2}(t-s),\lambda^{-1}(x-y))$. Here $w(t,x)$ may be taken to be $(1+t+|x|)^\sigma$ for sufficiently large $\sigma>0$. See \cite[Section 5 - Lemma 9]{MWb} for a proof of these support statements. 

If we define $R_{\epsilon}h(x) = \epsilon^{\frac{d}2 -1} h(\epsilon^{2}t, \epsilon x)$ (to be interpreted by integration against a test function if $X$ consists of a space of distributions), then one verifies that $R_{\epsilon}$ sends $X$ boundedly to itself and satisfies all of the conditions of Lemma \ref{sg} for either choice of norm above. Next we need to verify the mixing condition on an appropriate space. Again we give two separate paradigms for doing this.
\begin{enumerate} \item 
Using e.g. Lemma \ref{dualnorm} and Remark \ref{opdu}, one verifies that the solution of the equation has Cameron-Martin space given by the completion of $ \mathcal F(U)$ with respect to the norm $\|h\|_H:=\|\partial_t h- \partial_x^2 h\|_{L^2 (\mathbb R_+\times \mathbb R^d)}$. Indeed, note that $\mathbb E[(h,\phi)(h,\psi)] = \langle \phi,K*K^s*\psi\rangle_{L^2(\mathbb R_+\times \mathbb R)}$ where $K^s(t,x) = \frac1{\sqrt{2\pi |t|}}{e^{-x^2/2|t|}}1_{\{t\le 0\}} $. Since $K*K^s$ is the kernel for the inverse operator of $(\partial_t-\partial_x^2)(\partial_t+\partial_x^2)$, the claim is immediate from Lemma \ref{opdu}. Consequently by Proposition \ref{thm 2} (and the discussion afterwards) one finds that the set $\{(\log\log(1/\e))^{-1/2}R_{\epsilon}h\}_{\epsilon \in (0,1]}$ is precompact and that its set of cluster points in $X$ as $\epsilon \to 0$ equals the closed unit ball of $H$.

\item Rather than viewing $h$ as an intrinsic object we can view it as a chaos of order 1 over $\xi$, specifically $h= \psi(\xi):= K * \xi$ so that $\psi_{hom}(h) = K*h$. We may view $\xi$ as taking values in the Banach space $\tilde X$ given (for example) by the closure of $\mathcal F(U)$ under the norm $\|K*f\|_X$. If we define $\tilde R_\e f(t,x) := \e^{3/2}f(\e^2 t, \e x)$ then one uses Lemma \ref{sg} to verify the mixing condition required of $\tilde R_\e$ (considering functions with compact support contained in $U$). Furthermore one has the relation $R_\e \psi =\psi \tilde R_\e, $ thanks to the scaling property of the heat kernel. Thus one applies Theorem \ref{mainr} to obtain the Strassen law.
\end{enumerate}
\end{example}

This concludes the examples with first-order chaos, and now we move on to examples with higher chaos. Henceforth we shall not verify the calculation of the Cameron-Martin norm and the mixing condition as explicitly as we have done above, as it should be understood that this may be done similarly to above using e.g. results from Appendix B.

\begin{example}\label{2nd}
We give a simple example in the second chaos. Let $X=C([0,1],\mathbb R^2)$, take $\mu$ to be the law on $X$ of two-dimensional standard Brownian motion, which has Cameron-Martin space $H= \{(f_1,f_2)\in X: \int_0^1 (f_1'(t)^2 + f_2'(t)^2)dt<\infty\}.$ Let $Y=C[0,1]$ and consider the (discontinuous) map $\psi:X\to Y$

$$\psi(B_1,B_2) = \int_0^\bullet B_1(t)dB_2(t).$$
It is known that $\psi$ is a homogeneous chaos of order two, see e.g. \cite[Chapter 1]{Nua96}. For $f\in Y$ and $(f_1,f_2) \in X$ let us define
\begin{align*}
Q_\epsilon f(t)&:= \epsilon^{-1}f(\epsilon t),\\ R_\epsilon (f_1,f_2)(t)&:= (\epsilon^{-1/2}f_1(\epsilon t), \epsilon^{-1/2} f_2(\epsilon t)).
\end{align*}
Then it is clear that $\psi\circ R_\epsilon = Q_\epsilon \circ \psi$ a.s., and thus one can obtain the Strassen law (as $\epsilon \to 0$) for the family of processes $$\bigg\{\bigg((2\epsilon \log\log(1/\epsilon))^{-1} \int_0^{\epsilon t} B_1(s)dB_2(s)\bigg)_{t\in [0,1]}\bigg\}_{0<\epsilon< 1/e}. $$
The compact limit set is checked to be $\{\int_0^\bullet f_1(s)f_2'(s)ds: \int_0^1 f_1'(s)^2+f_2'(s)^2\leq 1\}. $ Note that one may strengthen the topology of $Y$ to the closure of smooth functions with respect to the Hölder norm of any exponent less than 1/2. See e.g. \cite[Proposition 3.4]{FH16} to see that the path has this Hölder regularity, and note that Lemma \ref{dualnorm} ensures that the compact limit set remains unchanged by changing the Banach spaces $X$ and $Y$ chosen to contain $(B_1,B_2)$ or $\psi(B_1,B_2)$ respectively.

Note that there is no additional difficulty in carrying out this same procedure for $k$-fold iterated integrals for any $k\in\mathbb N$, which recovers and generalizes \cite{ooo}.
\end{example}

\begin{example}
    Consider a standard space-time white noise $\xi$ on $\mathbb R_+\times \mathbb R$ and consider a $k$-fold stochastic integral of the form 
    \begin{equation}\label{ff}F(t,x):= \int_{t_1<...<t_k} \int_{\mathbb R^k} \prod_{j=1}^k p_{t_i-t_{i-1}}(x_i-x_{i-1})\xi(dt_k,dx_k)\cdots \xi(dt_1,dx_1),
    \end{equation}
    where $t_0=t$, $x_0=x,$ and $p_t(x) = (2\pi t)^{-1/2} e^{-x^2/2t}$ is the standard heat kernel. Such an expression appears when considering chaos expansions for the Hopf-Cole transform of the KPZ equation, see e.g. \cite{Wal86}. If we consider $\psi(\xi)=F$ as a homogeneous chaos of order $k$ \cite[Chapter 1]{Nua96}, then it turns out \cite{Wal86} that $\psi$ can take values in the space $\mathcal C^\alpha$ given by the closure of smooth functions in the Banach space of Hölder continuous functions on $B:=[0,1]\times [-1,1]$ (for instance), with norm
\begin{equation*}
    \|h\|_{\mathcal C^\alpha}:= \sup_{(t,x)\in B}|h(t,x)| + \sup_{(t,x)\neq (t',x)\in B} \frac{|h(t,x)-h(t',x')|}{|t-t'|^{\alpha/2} + |x-x'|^\alpha}.
\end{equation*} 
Furthermore, if one defines $(R_\e\xi)(t,x):= \e^{3/2}\xi(\e^2t,\e x)$, then by using the identity $cp_{c^2 t}(cx) = p_t(x)$ for $c>0$, one verifies that $\psi R_\e = Q_\e \psi$ where $Q_\e F(t,x) = \e^{-k/2}F(\e^2t, \e x)$. Consequently, with $F$ given by \eqref{ff}, the set of limit points of $\{(\log\log(1/\e))^{-k/2} \e^{-k/2} F(\e^2 t,\e x)\}$ is equal to the compact set of functions in $\mathcal C^\alpha$ which are of the form $$\int_{t_1<...<t_k} \int_{\mathbb R^k} \prod_{j=1}^k p_{t_i-t_{i-1}}(x_i-x_{i-1})f(t_k,x_k)\cdots f(t_1,x_1)dx_k\cdots dx_1dt_k\cdots dt_1,$$ where $\|f\|_{L^2(\mathbb R_+\times \mathbb R)}\leq 1.$
\end{example}

\begin{example} Let $b:\mathbb R^d \to \mathbb R^d$ be smooth and globally Lipchitz, and let $\sigma : \mathbb R^d\to \mathbb R^{d\times d}$ be smooth and globally Lipchitz. Let $B=(B_t)_{t\ge 0}$ be a standard Brownian motion. For $\delta>0$ we consider the Itô SDE in $\mathbb R^d$ given by 
\begin{equation}\label{Itôx}dX^\delta = b(X^\delta) dt + \delta \cdot \sigma(X^\delta)dB,
\end{equation}
with $X^\delta(0)=0$. Then a unique solution exists which is adapted to the filtration of $B$ \cite[Chapter XI]{ry99}. It is known for $d\ge 2$ that the measurable map $B\mapsto X^\delta$ is discontinuous in general. 

However we can factor it into a continuous map and a second-order chaos as follows. We fix $\alpha \in (1/3,1/2)$. We let $C^\alpha([0,1],\mathbb R^d)$ be the closure of smooth paths with respect to the $\alpha$-Hölder norm on paths $[0,1]\to \mathbb R^d$, and we let $C^\alpha([0,1]^2,\mathbb R^d\otimes \mathbb R^d)$ be the Banach space consisting of those functions $(s,t)\mapsto \mathbb F_{s,t}\in \mathbb R^d\otimes \mathbb R^d$ which lie in the closure of the set of smooth functions vanishing along the diagonal, with respect to the norm given by
$$\sup_{s\neq t} \frac{|\mathbb F_{s,t}|}{|t-s|^{2\alpha}},$$ where $|\cdot|$ can be any norm on $\mathbb R^d\otimes \mathbb R^d$. We then let $\mathcal M\subset C^\alpha([0,1],\mathbb R^d) \times C^\alpha([0,1]^2,\mathbb R^d\otimes \mathbb R^d)$ be the space of rough paths in the sense of \cite[Definition 2.1]{FH16}, that is all those pairs $(F, \mathbb F) \in C^\alpha([0,1],\mathbb R^d) \times C^\alpha([0,1]^2,\mathbb R^d\otimes \mathbb R^d)$ satisfying $$\mathbb F_{s,t} = \mathbb F_{s,u} + \mathbb F_{u,t} + (F_u-F_s)\otimes (F_s-F_t),\;\;\;\;\;\;\;\;\;\;\;0\leq s,t,u\leq 1.$$
Then there exists a \textbf{continuous} map $\Phi: \mathcal M \to C([0,1],\mathbb R^d)$ such that the solution of \eqref{Itôx} is given by
\begin{equation}\label{ilmap}X^\delta = \Phi \bigg( \delta B, \delta^2\int_\bullet^- (B_r-B_\bullet)\otimes dB_r\bigg),\end{equation} for all $\delta>0$, where explicitly the latter stochastic integral should be understood as $$\bigg(\int_\bullet^- (B_r-B_\bullet)\otimes dB_r\bigg)_{s,t}: = \sum_{i,j=1}^d \bigg( \int_s^t (B^i_r-B^i_s)dB^j_r\bigg)e_i \otimes e_j$$ where $\{e_i\}$ is the standard basis of $\mathbb R^d$. The fact that this second-order chaos is a.e. defined (albeit discontinuous as a function of $B$) and takes values in the Banach space $C^\alpha([0,1]^2,\mathbb R^d\otimes \mathbb R^d)$ for $\alpha \in (1/3,1/2)$ is proved in \cite[Proposition 3.4]{FH16}. This map $\Phi$ is called the Itô-Lyons map and its construction and structural properties are described in great detail in \cite[Chapter 8]{FH16}. 
More specifically, the continuity of $\Phi$ is proved as \cite[Theorem 8.5]{FH16} and the identity \eqref{ilmap} is proved as \cite[Theorem 9.1]{FH16}.

Now let us formulate a version of  Strassen's law for this SDE. If we let $R_\epsilon f:C^\alpha([0,1],\mathbb R^d)\to C^\alpha([0,1],\mathbb R^d)$ by $R_\e f(t) = \epsilon^{-1} f(\epsilon^2 t)$ then we can consider the sequence $X^\epsilon$ of Itô solutions to $$dX^\e = b(X^\e) dt + (\log\log(1/\e)) ^{-1/2}\sigma(X^\e)d\big( R_\e B\big).$$ 
Moreover, we have the second-order chaos $\psi: C^\alpha([0,1],\mathbb R^d) \to C^\alpha([0,1]^2,\mathbb R^d\otimes \mathbb R^d)$ given by $$\psi(B) = \int_\bullet^- (B_r-B_\bullet)\otimes dB_r.$$
We can then define $Q^1_\epsilon:=R_\epsilon:C^\alpha([0,1],\mathbb R^d)\to C^\alpha([0,1],\mathbb R^d)$ and we can also define $Q^2_\e: C^\alpha([0,1]^2,\mathbb R^d\otimes \mathbb R^d)\to C^\alpha([0,1],\mathbb R^d\otimes \mathbb R^d)$ by $Q^2_\e g(s,t) = \e^{-2} g(\e^2s,\e^2 t),$ then one verifies that the commutation assumption of Theorem \ref{mainr} holds, namely that $(R_\e B, \psi(R_\e B)) = (Q^1_\e B, Q^2_\e \psi(B)).$ 
The $Q^i$ are checked to be strongly continuous using Lemma \ref{sg}. We already remarked in an earlier example that $R_\e$ is strongly mixing, so by continuity of the map $\Phi$, we may then apply Theorem \ref{mainr}, as well as the result of Example \ref{2nd} to obtain that the set of limit points of $C([0,1],\mathbb R^d)$ of $\{X^\e\}_{\e \in (0,1]}$ as $\e\to 0$ is equal to the compact set $K$ of functions of the form $$\Phi \bigg( f, \int_\bullet^- (f(r)-f(\bullet))\otimes f'(r)dr\bigg),$$ as $f$ ranges throughout the closed unit ball of the Cameron-Martin space of $B$, namely $\{f\in C([0,1],\mathbb R^d): \int_0^1 \|f'(s)\|_{Eu}^2ds \leq 1\},$ where $\|x\|_{Eu}$ denotes the Euclidean norm of $x\in\mathbb R^d$. As above $(\bullet, -)$ denotes the argument of a general two-variable function here.

Now when $f$ is a smooth function from $[0,1]\to\mathbb R^d$, then the last expression inside of $\Phi$ is a smooth geometric rough path, and the Itô-Lyons map $\Phi$ acts on such a path by sending it to the solution of the classical ODE $\dot x = b(x) + \sigma(x)f$ as explained in \cite[Remark 8.4]{FH16}. We may then conclude that the compact limit set $K$ of $\{X^\e\}$ is equal to the closure in $C([0,1],\mathbb R^d)$ of the set of functions $x$ where $x(0)=0$ and $\dot x = b(x) + \sigma(x)f$ for some smooth path $f\in C([0,1],\mathbb R^d)$ such that $\int_0^1 \|f'(s)\|_{Eu}^2 \leq 1.$ 
\end{example}

\begin{example}[A counterexample]
    We conclude with an example where the semigroup $(R_\e)$ does \textit{not} extend to the full space $X$ in a strongly continuous fashion. Let $H=L^2(\mathbb T)$ where $\mathbb T=\mathbb R/\mathbb Z$. If we view $\mathbb T=[0,1]$ with the endpoints identified, we may sensibly define $(R_\e\xi)(x)= \e^{1/2}\xi(\e x)$ for $\xi \in H$. It is clear that this defines a bounded linear map on $H$ satisfying $R_\e R_\e^*=I$. Let $X=C^{-\alpha}(\mathbb T)$ where $\alpha>1/2$, which is defined to be the closure of smooth functions $f:\mathbb T\to \mathbb R$ under the norm given by $$\|f\|:= \sup_{\lambda\in (0,1]}\sup_{x \in \mathbb T} \sup_{\|\varphi\|_{C^1}\leq 1} \lambda^\alpha (f,\varphi^\lambda_x)_{L^2(\mathbb T)}$$ where $\varphi^\lambda_x(y) = \lambda^{-1}\varphi(\lambda^{-1}(y-x))$. Then the law of $\xi$ is supported on $X$ as shown in e.g. \cite[Section 2]{HL18}, thus $R_\e$ extends from $H$ to an a.e. linear measure-preserving operator on $X$ for each $\e>0$ by Lemma \ref{bij}. Nonetheless $R_\e$ does not extend to a bounded operator on the full space $X$ as it would violate the periodic boundary conditions. On one hand, by Theorem \ref{thm1} one still has a ``discrete-time" Strassen law along any subsequence of the form $\e_N=\alpha^N$ where $\alpha\in (0,1)$, but to even sensibly talk about the continuous version one would first need to \textit{simultaneously} define the uncountable collection $(R_\e)_{\e\in(0,1]}$ on a subset of full measure, which seems nontrivial. This illustrates that the assumptions in Theorem \ref{mainr} are necessary.
\end{example}

\subsection{Strassen law for the KPZ equation}

In this subsection we prove Theorem \ref{buss}. Although we focus on KPZ, the method described here is general and may apply to any SPDE whose solution map can be factored into a finite number of chaoses and a continuous part $$\xi \mapsto (\psi_1(\xi),...,\psi_m(\xi)) \stackrel{\text{continuous}}{\mapsto} \text{solution of SPDE},$$ which is always the case when using the theory of regularity structures as we do here. The only caveat is that it is difficult to come up with interesting examples of semigroups satisfying the mixing conditions unless the solution (in addition to admitting a factorization as above) is globally well-posed on the full space $\mathbb R^d$, for which KPZ (or rather its Hopf-Cole transform) seems to be the only example \cite{HL18, PR19}. For example, in \cite{MWb} the $\Phi_2^4$ equation is shown to be globally well-posed on $\mathbb R^2$ but it lacks the explicit construction of a continuous part of the solution map on the full space, which is needed in our context.
\\
\\
Let $\xi$ be a standard space-time white noise on $\mathbb R_+\times \mathbb R$ and define the following distribution-valued chaoses for $z\in \mathbb R_+\times \mathbb R$:
\begin{align*}\Pi_z^\xi\Xi(\psi)&:=\int_{\mathbb R_+\times \mathbb R}\psi(w)\xi(dw),\\ \Pi_z^\xi\big[\Xi \mathcal I[\Xi] \big](\psi)&:= \int_{(\mathbb R_+\times \mathbb R)^2} \psi(w) (K(w-a)-K(z-a))\xi(da)\xi(dw),\\ \Pi_z^\xi\big [\Xi \mathcal I[\Xi \mathcal I[\Xi]]\big](\psi)&:= \int_{(\mathbb R_+\times \mathbb R)^3} \psi(w)(K(w-a)-K(z-a)) (K(a-b)-K(z-b))\xi(db)\xi(da)\xi(dw),
\end{align*}
\begin{align*}\Pi_z^\xi\big [ \Xi \mathcal I[\Xi &\mathcal I[\Xi \mathcal I[\Xi]]]\big](\psi):= \int_{(\mathbb R_+\times \mathbb R)^4} \psi(w)(K(w-a)-K(z-a)) (K(a-b)-K(z-b)) \\& \cdot\big( K(b-c)-K(z-c) - (t_z-t_c)\partial_tK(z-c) - (x_z-x_c)\partial_xK(z-c)\big)\xi(dc)\xi(db)\xi(da)\xi(dw),
\end{align*}
\begin{equation*}
    \Pi_z^\xi\big[ \Xi \mathcal I[X_1\Xi]\big](\psi):= \int_{(\mathbb R_+\times \mathbb R)^2}\psi(w) K(w-a)(x_z-x_a)\xi(da)\xi(dw)\;\;\;\;\;\;\;\;\;\;\;\;\;\;\;\;\;\;\;\;\;\;\;\;\;\;\;\;\;\;\;
\end{equation*}
where $z=(t_z,x_z), a=(t_a,x_a),b=(t_b,x_b),$ and $c=(t_c,x_c)\in \mathbb R_+\times \mathbb R.$ We refer to e.g. \cite[Appendix A]{CW15} for a definition of the iterated stochastic integrals against space-time white noise on the full space. The requirement for these stochastic integrals to exist is that the integrands are $L^2$ functions jointly in all of the integration variables, which may be checked using the decay properties of the heat kernel. The usage of the abstract symbols on the left side will be explained shortly, but for now one should simply take the stochastic integrals on the right side as the definition of the objects on the left side, which are indexed by $\psi$ and $z$. For 
\begin{equation}\label{w}\tau\in \mathcal W:=\{\Xi, \Xi \mathcal I[\Xi], \Xi \mathcal I[\Xi \mathcal I[\Xi]], \Xi \mathcal I[\Xi\mathcal I[\Xi \mathcal I[\Xi]]], \Xi \mathcal I[X_1\Xi]\},\end{equation} we define the order $|\tau|$ to be $-3/2-\kappa, -1-2\kappa, -1/2-3\kappa, -4\kappa, -2\kappa$ respectively, where $\kappa\in(0,1/8)$ is arbitrary but fixed.

Now we want to find a Banach space for the above chaoses to be supported. To do this we follow a method used by \cite[Section 4]{HW13}. Clearly each $\Pi_z^\xi \tau$ is linear in $\psi$ as a map from $L^2(\mathbb R_+\times \mathbb R)\to L^2(\Omega,\mathcal F,P)$ where the latter denotes the underlying probability space of the noise $\xi.$ Thus one can expect to be supported on a space of distributions in space-time. In fact we will do this jointly in the $(z,\psi)$ coordinates, so that the $\Pi^\xi$ are viewed as a generalized function of two space-time coordinates $(z,z')$. 

\begin{defn}Fix a time horizon $T>0$ and set $\Lambda_T:=[0,T]\times \mathbb R$. We thus identify $\Lambda_T$ as our space-time, so that the space $\Lambda_T^2 = \Lambda_T\times\Lambda_T$ consists of pairs of space-time points $(z,z')$. For $j=1,2$ we say that a function $\Lambda_T^j\to \mathbb R$ is Schwartz if it is the restriction to $\Lambda_T^j$ of some Schwartz function from $\mathbb R^{2j}\to \mathbb R$. Given a Schwartz function $\Pi: \Lambda_T^2 \to \mathbb R$ and $\alpha<0$ one may define a norm $$\|\Pi\|:= \sup_{\lambda \in (0,1]} \sup_{\varphi\in B_r} \sup_{z\in \Lambda_T}w(z)^{-1} \lambda^{-\alpha} \int_{\Lambda_T}\Pi(z,z') \mathcal S^\lambda_z\varphi (z')dz',$$ where if $z = (s,y)$ and $z'=(t,x)$ then $\mathcal S^\lambda_z\varphi(z') = \lambda^{-3} \varphi(\lambda^{-2}(t-s),\lambda^{-1}(x-y))$, where $w$ is a weight function, and where $B_r$ is the set of smooth functions $\mathbb R^2\to\mathbb R$ supported on the closed unit ball with $C^r$ norm less than $1$ where $r:=-\lfloor \alpha \rfloor$. One may then define $E_w^\alpha$ to be the closure of Schwartz functions $\Pi$ under this norm (thus $E_w^\alpha$ implicitly depends on the time horizon $T$ which we have chosen). 
\end{defn}

Note that if $\Pi\in E^\alpha_w$, then for all $z\in\Lambda_T$ we can sensibly define $\Pi(z,\cdot)$ as a tempered distribution supported on $\Lambda_T$. This will be important for Theorem \ref{solmapkpz} Item \textit{(3)} below.

It turns out that the above symbols $\Pi^\xi\tau$ defined above are supported in the spaces $E_{w}^{|\tau|}$ for the appropriate choice of $w$ (see Theorem \ref{solmapkpz} just below). Assuming this fact, let us now define the appropriate semigroups on the spaces $E_\tau$. 
The semigroup acting on the noise is $R_\epsilon f(t,x) = \epsilon^{3/2} f(\epsilon^2t,\epsilon x).$ For $\tau \in \mathcal W$, define the $E^{|\tau|}_w$-valued homogeneous chaos $\psi_\tau(\xi):= \Pi^\xi\tau $ as given above. Then by using the fact that each of the functions $K(t,x), t\cdot \partial_tK(t,x), x\cdot\partial_xK(t,x)$ are invariant under the scaling $f(t,x)\to cf(c^2t,cx)$ for $c>0$, one checks that the following five semigroups $Q^\tau$ on $E^{\tau}_w$ will satisfy the commutation relation $Q^\tau_\epsilon \circ \psi_\tau = \psi_\tau \circ R_\epsilon$ which is needed to apply Theorem \ref{mainr}:

$$Q^\Xi_\epsilon f(t,x,s,y) = \epsilon^{3/2} f(
\epsilon^2 t,\epsilon x, \epsilon^2 s, \epsilon y),$$
$$Q_\epsilon^{\Xi \mathcal I[\Xi]}f(t,x,s,y) = \epsilon f(\epsilon^2 t, \epsilon x, \epsilon^2s, \epsilon y),$$
$$Q^{\Xi \mathcal I[\Xi \mathcal I[\Xi]]}_\epsilon f(t,x,s,y) = \epsilon^{1/2} f(\epsilon^2 t, \epsilon x, \epsilon^2 s, \epsilon y),$$
$$Q_\epsilon^{\Xi \mathcal I[\Xi \mathcal I[\Xi \mathcal I[\Xi]]]}f(t,x,s,y) =  f(\epsilon^2 t, \epsilon x, \epsilon^2 s, \epsilon y),$$
$$Q_\epsilon^{\Xi \mathcal I[X_1\Xi]}f(t,x,s,y) = f(\epsilon^2 t, \epsilon x, \epsilon^2 s, \epsilon y)$$
where the $(t,x)$ stands for the $z$ variable and the $(s,y)$ stands for the $w$ variable in the integrals above (note that the $f$ appearing here will be distributional in the latter variable and thus, as usual, the rescaling of coordinates needs to be interpreted by integrating against test functions). In other words one simply has $Q^\tau_\epsilon f(t,x,s,y) = \epsilon^{-|\tau|_0}f(\epsilon^2 t, \epsilon x, \epsilon^2s, \epsilon y)$ where $|\tau|_0$ denotes the order of $\tau$ without the $\kappa$ term. 

We henceforth fix $s,y>0$ and $\alpha \in (0,1/2-4\kappa)$ where $\kappa$ is the same as specified after \eqref{w}. We will let $\mathcal C^\alpha_{s,y}$ denote the closure of smooth functions in the Banach space of Hölder continuous functions on $B:=[0,s]\times [-y,y]$, whose norm is given by 
\begin{equation}\label{holderspace}
    \|h\|_{\mathcal C^\alpha_{s,y}}:= \sup_{(t,x)\in B}|h(t,x)| + \sup_{(t,x)\neq (t',x)\in B} \frac{|h(t,x)-h(t',x')|}{|t-t'|^{\alpha/2} + |x-x'|^\alpha}.
\end{equation} We will now prove Theorem \ref{buss}. The main idea is to first note that in the notation of that theorem, we see that the function $\tilde h(t,x):= C_\e^{-1} \epsilon^{-1/2} h^\e (\epsilon^2 t,\e x)$ is the Hopf-Cole solution of $$\partial_t \tilde h(t,x) = \partial_x^2 \tilde h(t,x) + (\partial_x \tilde h(t,x))^2 + (\log\log(1/\e))^{-1/2}(R_\e \xi)(t,x),$$ where $R_\e$ is the measure-preserving operation defined above. Therefore when combined with Theorem \ref{mainr}, the following result clearly implies Theorem \ref{buss}. Its proof is based entirely on the extensive framework and estimates developed in \cite{HP15, HL18}. Moreover, taking the logarithm poses no issue since we are working on a compact domain and the solution is strictly positive.

\begin{thm}[\cite{HL18,HP15}]\label{solmapkpz} 
\begin{enumerate}
    \item The above symbols $\Pi^\xi\tau$ defined above are supported in the spaces $E_{w}^{|\tau|}$ for the appropriate choice of weight $w$. Furthermore the semigroups $Q^\tau$ as defined above are strongly continuous on the spaces $E^{|\tau|}_{w}$.
    
    \item For all $\delta>0$ replace $\xi$ by $\delta \xi$ in the above stochastic integrals defining $\Pi^\xi\tau$, and denote the resulting 5-tuple of objects as $(\Pi^{\delta\xi}\tau)_{\tau \in \mathcal W}$. There exists a (nonlinear) closed subset $\mathcal X$ 
    of $\bigoplus_{\tau \in \mathcal W} E^{|\tau|}_{w}$ with the property that for all $\delta>0$ the 5-tuple $(\Pi^{\delta\xi}\tau)_{\tau \in \mathcal W}$ is supported on $\mathcal X$, and moreover there exists a deterministic continuous map $\Phi: \mathcal X \to \mathcal C_{s,y}^\alpha$ such that 
    $\Phi$ sends $(\Pi^{\delta\xi}\tau)_{\tau \in \mathcal W}$ to the Itô solution of the SPDE given by $\partial_t z = \partial_x^2z + \delta z\xi$, with initial data $z(0,x)=1$.
    
\item For $\tau \in \mathcal W$, define the $E^{|\tau|}_w$-valued homogeneous chaos $\psi_\tau(\xi):= \Pi^\xi\tau $ as above. Let $(\psi_\tau)_{hom}(z;f)$ be defined as $(\psi_\tau)_{hom}(f)$ evaluated at $z\in \Lambda_T$ in the first coordinate. Then for all space-time Schwartz functions $f$, the function $(\psi_\tau)_{hom}(z;f)$ is given by the expression for $\Pi_z^\xi\tau$ but with $\xi(dw),\xi(da),\xi(db),\xi(dc)$ respectively replaced by $f(w)dw, f(a)da, f(b)db, f(c)dc$ (so these are now classical integrals, not stochastic integrals as above). 

\item For all space-time Schwartz functions $f$ on $\Lambda_T$, the map $\Phi$ sends $\big((\psi_\tau)_{hom}(f)\big)_{\tau \in \mathcal W}$ to the classical solution of the deterministic equation $\partial_tz = \partial_x^2 z + zf,$ with initial data $z(0,x)=1$. 
 \end{enumerate}
\end{thm}

\begin{proof}
   \textit{(1)} Let $w(t,x) = e^{a(1+|x|)}$ where $a>0$ is arbitrary.  The fact that the symbols $\Pi^\xi \tau$ are supported on $E^{|\tau|}_{w}$ follows from \cite[Theorem 5.3]{HL18} by taking $p=\infty$, see Remark 2.7 there (note that since our initial data is zero which is a smooth function, we do not even need to consider finite $p$ in their results, which is done only for the sake of considering singular initial data). The fact that the operator semigroups $Q^\tau$ are strongly continuous on these spaces is follows from Lemma \ref{sg}.
\\
\\
   \textit{(2)} One defines the regularity structure $(T,A,G)$ where $T$ is the graded vector space generated by the list of twelve symbols in \cite[Figure 1]{HL18}, and $A$ is the respective list of homogeneities. Note that the list of symbols is precisely those symbols $\tau$ appearing in $\mathcal W$ as defined in \eqref{w} plus symbols of the form $\mathcal I \tau$ with $\tau \in \mathcal W$ as well as $\Xi X_i$ and $\mathcal I(\Xi X_i).$ The group $G$ will be described later.
   
   We are going to let $\mathcal X$ denote the space of those functions $\Pi\in \bigoplus_{\tau \in \mathcal W} E^{|\tau|}_{w}$ which extend to admissible models in the sense of \cite[Definition 2.2]{HL18} (such $\Pi$ are also called ``minimal models" in \cite{HW13}). For any map $\Pi$ satisfying the requirements \cite[(2.2) and (2.5)]{HL18} of an admissible model, there is a way to define a corresponding family of invertible transformations $\Gamma_{xy}$ of $T$ such that one has the relations $\Pi_x \Gamma_{xy} = \Pi_y$ and $\Gamma_{xy}\Gamma_{yz} = \Gamma_{xz}$ for all $x,y,z \in \Lambda_T$, as well as \cite[(2.3)]{HL18}. The first condition and the fact that $\Pi_xX^k (z)= (z-x)^k$ forces $\Gamma_{xy}$ to act on elements $\mathcal I[\tau]$ by the formula $$\Gamma_{xy}\mathcal I[\tau] = \mathcal I[\tau]+\sum_{|k|_s\leq |\tau|+2} \frac{X^k}{k!} \int_{\mathbb R^{d+1}} \big( \partial^kK(y-a)-\partial^k K (x-a) \big) (\Pi_x\tau)(da)$$$$ \;\;\;\;\;\;\;\;\;\;\;\;\;\;\;\;\;\;\;\;\;\;\;+ \sum_{|k|_s\le |\tau|+2} \frac{(X+y-x)^k-X^k}{k!} \int_{\mathbb R^{d+1}} \partial^k K(y-a)(\Pi_x\tau)(da),$$ where the sum is over multi-indices $k=(k_1,k_2)$ and $|k|_{\mathfrak s} = 2k_1+k_2.$ Furthermore it is clear that $\Gamma_{xy}\Xi=\Xi$. Finally one has $\Gamma_{xy}(\Xi X^k\mathcal I[\tau]) = \Xi X^k \Gamma_{xy}\mathcal I[\tau]$ and $\Gamma_{xy} X^k = (X+y-x)^k.$ The structure group $G$ may then be described as the group of all $\Gamma_{x,z}$ as $x,z$ range throughout $\Lambda_T.$

   Since the five symbols $\Pi^{\delta\xi}\tau$ ($\tau\in \mathcal W$) may be checked to verify \cite[(2.2) and (2.5)]{HL18} for all $\tau \in \mathcal W$, and since $\Pi \tau$ for $\tau\notin \mathcal W$ may be obtained uniquely and continuously as a function of these five symbols by the same rules \cite[(2.2) and (2.5)]{HL18}, we then conclude that those five symbols indeed constitute an element of $\mathcal X$. Now the existence and local Lipchitz continuity of the solution map $\Phi$ is done in \cite[Theorem 5.2]{HL18}, and the fact that $\Phi$ sends $(\Pi^{\delta\xi}\tau)_{\tau \in \mathcal W}$ to the Itô solution of the SPDE given by $\partial_t z = \partial_x^2z + \delta z\xi$, with initial data $z(0,x)=1$, is stated in \cite[Theorem 1.1]{HL18} (the proof is not given there but rather it is stated that the proof is nearly identical to \cite[Theorem 6.2]{HP15} which is done in a compact setting).
\\
\\
\textit{(3)} Suppose $(X,H,\mu)$ is any abstract Wiener space, $Y$ is another Banach space. Suppose that $\psi:X\to Y$ is any ``simple" Gaussian chaos of order $k$ given by $x\mapsto (H_k\circ g)(x)y$ where $y\in Y$ is fixed, $g\in X^*$ is fixed, and $H_k$ denotes the $k^{th}$ Hermite polynomial. Then one may check using \eqref{homm} that $\psi_{hom}(h) =g(h)^ky$. Since ``hom" is an additive operation and since any chaos of order $k$ can be approximated by finite linear combinations of simple chaoses of order $k$ (Lemma \ref{approximation lemma'}), the claim then follows from Lemma \ref{hom is cont'} since the space-time white noise $\xi$ (when viewed as an element of the space $E^{|\Xi|}_w)$ has Cameron-Martin space given by $L^2([0,T]\times \mathbb R)$, see e.g. Lemma \ref{dualnorm}.
\\
\\
\textit{(4)} We note that the 5-tuple $\big((\psi_\tau)_{hom}(\cdot;f)\big)_{\tau \in \mathcal W}$ extends to a canonical model in the sense of \cite[Remark 2.3]{HL18}. This precisely means that \cite[(5.3)]{HL18} is satisfied except that $c_\epsilon$ and $c^{(1)}_\epsilon$ are both replaced by $0$. Then following through that proof immediately gives the claim.
\end{proof}

\section{Further generalizations} Although Theorem \ref{mainr} is fairly general and widely applicable, there are interesting examples where it does not apply. Here we study such examples, giving further gerenalizations such as Theorems \ref{gen} and \ref{repar} below.

\subsection{Noise-smoothing and non-measure-preserving systems}\label{mollify}

In this section we will further generalize the results of the previous sections. This will culminate in Theorems \ref{contraction 3} and \ref{gen}, with a specific SPDE example given in Theorem \ref{mollkpz}.

One may want to consider proving Strassen's law for a family of processes that is being diffusively scaled while being simultaneously approximated by mollified noise, similar to what was done in \cite{HW13}. Specifically consider $\delta = \delta(\epsilon)$ and look at the family of processes $h^\e$ defined with smooth noises at scale $\delta(\e)$, i.e., letting $L_2:= \log \log$ we want to consider the ``smooth KPZ equation" $$\partial_t h^\e = \partial_x^2 h^\e +((\partial_xh^\e)^2- C_{\delta(\e)/\e} ) +(L_2(1/\e))^{-1/2} \epsilon^{3/2} \xi_{\delta(\e)}(\epsilon^2 t,\epsilon x),$$ where $\xi_\delta= \xi * \varphi_\delta$ where $*$ is space-time convolution and $\varphi_\delta(t,x) = \delta^{-3}\varphi(\delta^{-2}t,\delta^{-1}x)$ for a smooth nonnegative compactly supported function $\varphi:\mathbb R^2\to\mathbb R$ which integrates to 1. For an appropriate choice of divergent constants $C_\e$, the goal is to try to find the set of limit points of the sequence of functions given by $h^\e(\epsilon^2 t, \epsilon x),$ and in particular show that the compact limit set is the same as in Theorem \ref{buss}. The main result of this section will be rather intuitive: the limit points in $\mathcal C^\alpha_{s,y}$ coincide with those found previously if $\delta(\epsilon) = \epsilon^{1+u}$ for some $u>0$, they depend on the mollifier if $\delta(\e) = \e$, and it is a trivial set consisting of a one-dimensional family of functions if $\delta(\e) = \e^{1-u}$. To prove this, we need some preliminaries. 

\begin{defn} Consider a strongly continuous (multiplicative) semigroup $(R_\e)_{\e\in (0,1]}$ on a Banach space $X$. Writing $S_t := R_{e^{-t}}$ note that $\|S_{t+s}\| \leq \|S_t\|\|S_s\|$, by Fekete's subadditive lemma we know that the following quantity exists as a real number $$\beta(R):= \lim_{t\to \infty} t^{-1}\log\|S_t\|= \inf_{t>0} t^{-1} \log\|S_t\|.$$ 
\end{defn}

In other words, $\|R_\e\|_{X\to X} \leq \e^{-\beta(R)-o_\e(1)},$ and $\beta(R)$ is the optimal such exponent (here $o_\e(1)$ is some non-negative function of $\e\in(0,1]$ which tends to zero as $\e\to 0)$. Now to illustrate our results, we will consider a simpler setting first. Suppose we have a Gaussian measure $\mu$ on $X$ such that $R_\e$ is measure-preserving and satisfies the strong mixing condition on the Cameron-Martin space $H$. Suppose we are given a family of bounded linear operators $A_\e$ on $X$ (where $(x,\e)\mapsto A_\e(x)$ is Borel-measurable, but $A_\e$ is not necessarily a semigroup). Suppose that there exists a Banach space $Y$ and constants $C,\gamma>0$ such that 
\begin{enumerate}
    \item $Y$ embeds continuously into $X$.
    
    \item $\mu(Y)=1.$
    
    \item $\|A_\e y- y\|_X \leq C \epsilon^{\beta(R)+\gamma}$ for all $y$ with $\|y\|_Y\leq 1$.
\end{enumerate}
Then we have the bounds $\|R_\e A_\e x- R_\e x \|_X \leq \|R_\e\|\|A_\e x - x\| \leq \e^{-\beta(R)-o_\e(1)} \|A_\e x-x\|$. Now for all $x$ in a set of full measure (namely $x\in Y$) we can bound $\|A_\e x-x\| \leq C\|x\|_Y \e^{\beta(R)+\gamma}.$ Consequently, $$\|R_\e A_\e x - R_\e x\|_X \leq C\e^{\gamma-o_\e(1)} \|x\|_Y,$$ and thus as long as $x\in Y$ (which is almost every $x$) then the set of limit points of $(\log\log(1/\e))^{-1/2} R_\e A_\e x$ coincides with those of $(\log\log(1/\e))^{-1/2} R_\e x$, namely $B(H)$.

This gives the seemingly obvious fact that the Strassen's law still holds as long as we mollify fast enough relative to the rescaling operation $(R_\e)$. Consider the simple example of two-sided Brownian motion for instance, say Wiener measure on $X = C[-1,1]$. Let $A_\e f(x) = f *\phi_\e$ where $\phi$ is a smooth even mollifier and $\phi_\e(x) = \e^{-1}\phi(\e^{-1}x)$ (by convention we let $A_\e f$ be constant on each of $[-1,-1+\e]$ and $[1-\e,1]$). Let $R_\e f(x) = \epsilon^{-u/2}f(\epsilon^u x)$ where $u>0$, so that $\beta(R) = u/2$. If $u<1$ then we may define $Y$ to be the closure of smooth functions with respect to the Hölder norm of exponent $(u+1)/4$ and we know that this space supports Wiener measure almost surely. Moreover an easy computation shows that $\|y*\phi^\e - y\|_{C[-1,1]} \leq C\epsilon^{(u+1)/4}$ as long as $\|y\|_Y\leq 1$ (where $C=\int_{\mathbb R} \phi(v)|v|^{(u+1)/4} dv$). Consequently by the discussion in the previous paragraph, one obtains the same Strassen law as Brownian motion for the family $R_\e A_\e$ on $C[-1,1]$. On the other hand if $u=1$ then the limit points will be the mollifier-dependent compact set $\{ \phi * f: f(0)=0, \|f'\|_{L^2[-1,1]}\le 1\},$ since one has that $R_\e A_\e f = \phi * R_\e f$ in this case. Thus we find that $u<1$ is necessary to obtain a nontrivial limit set that coincides with the non-mollified case. For $u>1$ the smoothing dominates the rescaling, so one obtains a trivial limit set consisting only of constant functions of absolute value bounded above by 1, as may be checked by hand.

We thus formulate the following abstract result:

\begin{thm}\label{contraction 3}
Let $(\B,\H,\mu)$ be an abstract Wiener space, let $(R_\e)_{\e\in(0,1]}$ be a family of Borel-measurable a.e. linear maps from $\B\to \B$ which are measure-preserving and strongly mixing, and let $T^i:\B\to Y_i$ be homogeneous of degree $k_i$ for $1\le i \le m$. Suppose that there exist strongly continuous semigroups $(Q^i_\e)_{t\ge 0}$ operators from $Y_i\to Y_i$ for $1\le i \le m$ with the property that 
$T^i\circ R_\e = Q^i_\e \circ T^i$, $\mu$-a.e. for all $\e\in(0,1]$. Suppose that $J_\e^i:X\to Y_i$ is a family of measurable maps such that \textbf{at least one} of the following two conditions holds \begin{enumerate}
\item There exist $\gamma>0$ and measurable functions $C^i:X\to \mathbb R_+$ such that $\|J_\e^i ( x )- T^i (x)\|_{Y_i} \leq C^i(x) \e^{\beta(Q^i)+\gamma}$ for all $\e\in (0,1].$

\item Letting $J^i_0:=T^i$, assume that $J^i_\e$ is a homogeneous chaos of order $k_i\in \mathbb N$ and there exist $C,\gamma>0$ such that $\int_X\|J_\e^i ( x )- J^i_{\e'} (x)\|_{Y_i} \;\mu(dx) \leq C|\e-\e'|^{\beta(Q^i)+\gamma},$ for all $\e,\e'\in[0,1].$
\end{enumerate}
Let $\mathcal M\subset Y_1\times \cdots Y_m$ be a closed subset, such that the semigroup $Q_\e^1 \oplus \cdots \oplus Q_\e^m$ sends $\mathcal M$ to itself, and moreover for all $\delta, \e>0$ assume that $\mu(\{x\in X: (\delta^{k_1} J^1_\e ( x),...,\delta^{k_m} J^m_\e(x)) \in \mathcal M\})=1.$ 

Then for all $\delta>0$ we have that $\mu(\{x\in X: (\delta^{k_1} T^1(x),...,\delta^{k_m} T^m(x)) \in \mathcal M\})=1,$ and furthermore the compact set $K:=\{(T^1_{\mathrm{hom}}(h),...,T^m_{\mathrm{hom}}(h)):h \in B(H)\}$ is necessarily contained in $\mathcal M$. Moreover, for any Banach space $Z$ and any $\Phi:\mathcal M \to Z$ that is uniformly continuous on bounded sets, the set of cluster points at zero of
$$\big\{ \Phi\big( (2\log \log(1/\e))^{-k_1/2} Q^1_\e J^1_\e( x),\;.\;.\;.\;,(2\log \log(1/\e))^{-k_m/2}Q^m_\e J^m_\e ( x)\big): \e \in  (0,1]\}$$ is almost surely equal to $\Phi(K).$
\end{thm}

We recover the result of the discussion above (for first order chaos) when we set $m=1$ with $X=Y_1=Z=\mathcal M$ and $J^1_\e = A_\e$ and $Q^i_\e = R_\e$ and $\Phi = I$ and $C^1(x) =C \|x\|_Y.$ Note also that we do \textit{not} make the requirement that $Q^i_\e J^i_\e = J^i_\e R_\e$ and in general this will be false. Finally we remark that when $J_\e^i = T^i$ the above theorem recovers Theorem \ref{contraction 2}, at least in the case that $\Phi$ is uniformly continuous and $\mathcal M$ is invariant under the semigroup (which is usually true in practice). 

\begin{proof} First let us show that $\mu(\{x\in X: (\delta^{k_1} T^1(x),...,\delta^{k_m} T^m(x)) \in \mathcal M\})=1.$ Note that, regardless of whether \textit{(1)} or \textit{(2)} holds, we have that $\|J_{2^{-N}}^i ( x )- T^i (x)\|_{Y_i}\to 0$ as $N\to\infty$ for almost every $x\in X$ (indeed $\beta(Q^i)$ must be strictly positive under the assumptions of the theorem, else $\|Q_\e^i\|_{Y_i\to Y_i}\leq 1$ so that $(\log\log(1/\e))^{-k_i/2}\|Q^i_\e T^i(x)\|_{Y_i} \to 0$, which contradicts the result of Theorem \ref{ar2}). Consequently, $(\delta^{k_1} T^1(x),...,\delta^{k_m} T^m(x))$ is the almost sure limit of $(\delta^{k_1} J^1_{2^{-N}} ( x),...,\delta^{k_m} J^m_{2^{-N}}(x))$. Since the latter is assumed to be an element of $\mathcal M$ and since $\mathcal M$ is a closed subset, we conclude that the former is also an almost sure element of $\mathcal M$.

The proof that $K$ is contained in $\mathcal M$ assuming that $\mu(\{x\in X: (\delta^{k_1} T^1(x),...,\delta^{k_m} T^m(x)) \in \mathcal M\})=1$ for all $\delta>0$ is given in the first paragraph of the proof of Corollary \ref{contraction}.

Now assume Condition \textit{(1)} holds true. One makes the bound $$\|Q^i_\e J^i_\e (x) -  Q^i_\e T^i(x)\| \leq \|Q^i_\e \|_{Y^i \to Y^i} \|J^i_\e (x) - T^i (x)\| \leq \| Q^i_\e \|_{Y^i \to Y^i} C^i(x)  \e^{\beta(Q^i)+\gamma} \leq C^i(x)\e^{\gamma-o_\e(1)}$$ where we used by definition of $\beta(Q^i)$ that $\|Q^i_\e\| \e^{\beta(Q^i)} \leq \e^{-o_\e(1)}$ (here $o_\e(1)$ is some non-negative function of $\e\in(0,1]$ which tends to zero as $\e\to 0)$. Multiplying by $(\log\log(1/\e))^{-k_i/2}$ on both sides, we see that the set of limit points of the desired sequence must coincide with that of Theorem \ref{mainr} since the difference tends to zero as an element of $Z$. Note that one needs the additional assumption of \textit{uniform} continuity of $\Phi$ on bounded subsets of $\mathcal M$ to argue this last part.

Now assume Condition \textit{(2)} holds true instead. Since $J^i_\e$ is a homogeneous chaos, \eqref{hyp} gives $$\bigg(\int_X\|J_\e^i ( x )- J^i_{\e'} (x)\|_{Y_i}^p \mu(dx)\bigg)^{1/p} \leq C_p|\e-\e'|^{\beta(Q^i)+\gamma}$$ for all $p\ge 1$. By Kolmogorov-Chentsov, for a.e. $x\in X$ it is clear that $\epsilon\mapsto J^i_\e(x)$ is Hölder continuous on $[0,1]$ of any desired exponent strictly less than $\beta(Q^i)+\gamma$. In particular, we may conclude that Condition \textit{(1)} holds with $\gamma$ replaced by $\gamma/2$. 
\end{proof}

Now let us argue how this result allows us to obtain the Strassen law for a version of KPZ driven by a smooth noise mollified in both time and space. Let $\xi$ be a standard space-time white noise on $\mathbb R_+\times \mathbb R$ and define the following distribution-valued chaoses for $z\in \mathbb R_+\times \mathbb R$ and $\delta>0$:
\begin{align*}\Pi_z^{\xi,\e}\Xi(\psi)&:=\int_{\mathbb R_+\times \mathbb R}\psi(w)\xi_\e(dw),\\ \Pi_z^{\xi,\e}\big[\Xi \mathcal I[\Xi] \big](\psi)&:= \int_{(\mathbb R_+\times \mathbb R)^2} \psi(w) (K(w-a)-K(z-a))\xi_\e (da)\diamond\xi_\e (dw),\\ \Pi_z^{\xi,\e}\big [\Xi \mathcal I[\Xi \mathcal I[\Xi]]\big](\psi)&:= \int_{(\mathbb R_+\times \mathbb R)^3} \psi(w)(K(w-a)-K(z-a)) (K(a-b)-K(z-b))\xi_\e (db)\diamond\xi_\e(da)\diamond\xi_\e(dw),
\end{align*}
\begin{align*}\Pi_z^{\xi,\e}&\big [ \Xi \mathcal I[\Xi \mathcal I[\Xi \mathcal I[\Xi]]]\big](\psi):= \int_{(\mathbb R_+\times \mathbb R)^4} \psi(w)(K(w-a)-K(z-a)) (K(a-b)-K(z-b)) \\& \cdot\big( K(b-c)-K(z-c) - (t_z-t_c)\partial_tK(z-c) - (x_z-x_c)\partial_xK(z-c)\big)\xi_\e(dc)\diamond\xi_\e(db)\diamond\xi_\e(da)\diamond\xi_\e(dw),
\end{align*}
\begin{equation*}
    \Pi_z^{\xi,\e}\big[ \Xi \mathcal I[X_1\Xi]\big](\psi):= \int_{(\mathbb R_+\times \mathbb R)^2}\psi(w) K(w-a)(x_z-x_a)\xi_\e(da)\diamond\xi_\e(dw)\;\;\;\;\;\;\;\;\;\;\;\;\;\;\;\;\;\;\;\;\;\;\;\;\;\;\;\;\;\;\;
\end{equation*}
where $z=(t_z,x_z), a=(t_a,x_a),b=(t_b,x_b),$ and $c=(t_c,x_c)\in \mathbb R_+\times \mathbb R.$ The ``$\diamond$" products above should be interpreted as follows: first one interprets the integrals as classical integrals against the smooth noise $\xi_\delta$, then one projects the resultant object onto the homogeneous chaos of order given by the number of times that $\xi_\delta$ appears in the integral. For $\tau \in \mathcal W$ from \eqref{w} we let $k_\tau$ be the number of instances of $\Xi$ in the symbol $\tau$. We then have the following result.

\begin{prop}\label{410}
    Let $\mathcal X$ and $\Phi$ be as in Theorem \ref{solmapkpz}, and let $\mathcal W$ be as in \eqref{w}. Then for all $\delta,\e>0$ the 5-tuple $(\delta^{k_\tau}\Pi^{\xi,\e}\tau)_{\tau \in \mathcal W}$ is almost surely an element of $\mathcal X$. Furthermore $\Phi\big ((\delta^{k_\tau}\Pi^{\xi,\e}\tau)_{\tau \in \mathcal W}\big)$ is given by the solution to the classical equation $$\partial_t z = \partial_x^2 z + z(\delta\xi_\e-C_{\e,\delta}),$$ with $z(0,x)=1$ and where $$C_{\e,\delta} = c\delta^2\e^{-1} + \delta^4c_\e'$$ where $c$ is a constant depending only on the mollifier (not $\e,\delta)$ and where $c_\e' $ may depend on both $\e$ and the mollifier but is convergent to a finite constant as $\e\to 0$.
\end{prop}

\begin{proof}
    This is proved in \cite[Theorem 5.3]{HL18}. The constants are explicitly written out in \cite[(1.1)]{HL18}, and their $\e\to 0$ asymptotics are explained in the subsequent paragraphs.
\end{proof}

With this established, we can prove the following theorem, the main result of this subsection.

\begin{thm}\label{mollkpz}
    Let $\varphi_\delta(t,x) = \delta^{-3}\varphi(\delta^{-2}t,\delta^{-1}x)$ for a smooth nonnegative compactly supported function $\varphi:\mathbb R^2\to\mathbb R$ which integrates to 1. Consider the sequence of solutions to the classical PDE $$\partial_t h(t,x) = \partial_x^2 h (t,x) +(\partial_xh(t,x))^2-C_{\e^{s-1},\delta(\e)}+\delta(\e) \e^{3/2} \xi_{\e^s}(\e^{2} t, \e x)  $$ where $h(0,x)=0$ and $\xi_\e = \xi * \varphi_{\e}$ and $\delta(\e) = (\log\log(1/\e))^{-1/2}$. If $s>1$ then the set of cluster points in $\mathcal C^{\alpha}_{t,y}$ as $\e\to 0$ is equal to the same compact limit set $K$ as in Theorem \ref{buss}, for any $\alpha<1/2$.
\end{thm}

\begin{proof}
Let $R_\e$ and $Q^\tau_\e$ ($\tau \in \mathcal W)$ be the same multiplicative semigroups as defined in the previous subsection. We define $J^\tau_\e(\xi) $ to be equal to $\Pi^{\xi,\e}\tau$. We also fix some exponent $s>0$. We claim for each $\e>0$ that 
$v:= \Phi\big( (\delta^{k_\tau} Q^\tau_{\e^s} J^\tau_{\e^s} \xi )_{\tau \in \mathcal W}\big)$ is the solution of the equation 
$$\partial_t v(t,x) = \partial_x^2 v (t,x) +v(\delta \e^{3/2} \xi_{\e^s}(\e^{2} t, \e x) -C_{\e^{s-1},\delta}) $$ where $v(0,x)=1$ and $\xi_\e = \xi * \varphi_{\e}$. Indeed one verifies that $Q^\tau_{\e^s} J^\tau_{\e^s} \xi = J^\tau_{\e^{s-1}} (R_{\e} \xi)$, so that the claim is immediate from Proposition \ref{410} and the fact that $\big((R_\e \xi)*\varphi_{\e^{s-1}}\big)(t,x) = \e^{3/2} \xi_{\e^s}(\e^{2} t, \e x)$ (note that this calculation shows why we need $s>1$).

Now to finish our analysis of the mollified equation, one needs as in assumption (2) of Theorem \ref{contraction 3} to show a bound of the type $\mathbb E\|J^\tau_\e (\xi) - J^\tau_{\e'}(\xi) \|_{E^{|\tau|}_w} \leq C |\e-\e'|^{\beta(Q^\tau)+\gamma}$. We first remark that for the semigroups $Q^\tau$ on the spaces $E^{|\tau|}_w$ one has $\beta(Q^\tau) = k_\tau \kappa$ where $\kappa$ on the right side is the same as in the previous subsection, which may be verified by hand. By choosing $\kappa, \gamma$ to be sufficiently small relative to the desired exponent $\alpha$ from the theorem statement, the desired bound follows (as noted in \cite[Theorem 5.3]{HL18}) by following the logic in \cite[Theorem 4.5]{HP15} (see also \cite[(3.7) and (3.8)]{temp}). By taking a logarithm and using compactness of the domain, we obtain the claim.
\end{proof}

We remark that mollification at scale $\e^s$ is not the only example of where Theorem \ref{contraction 3} is useful. One can also use it, for instance, when looking at a noise $\xi$ that is spatially periodic with some large period $\epsilon^{-1}$, by defining the chaoses $J^i_\e$ to be the spatially periodized version of the chaoses $T^i$ (which may be checked to satisfy a bound of the desired form \textit{(2)} in Theorem \ref{contraction 3}). This would be useful in showing a Strassen law for KPZ (and potentially other singular models) on periodic domains such as a torus.

While the preceding result illustrates the use of Theorem \ref{contraction 3}, we note that it can be generalized even further by considering some family of maps $\Phi_\e:\mathcal M \to Z$ where $\e\in (0,1]$ and then imposing suitably strong convergence conditions on $\Phi_\e$ as $\e\to 0$. 
Additionally, we remark that in Condition \textit{(1)} the chaoses $J^i_\e$ do \textit{not} need to be assumed to be deterministic functions of the Gaussian noise $x$. 
In particular we have the following result which can be viewed as a generalization of all of other results thus far, whose interest will be explained shortly. 

\begin{thm}\label{gen}Let $(\B,\H,\mu)$ be an abstract Wiener space, let $(R_\e)_{\e\in(0,1]}$ be a family of Borel-measurable a.e. linear maps from $\B\to \B$ which are measure-preserving and strongly mixing, and let $T^i:\B\to Y_i$ be homogeneous of degree $k_i$ for $1\le i \le m$. Suppose that there exist strongly continuous semigroups $(Q^i_\e)_{t\ge 0}$ operators from $Y_i\to Y_i$ for $1\le i \le m$ with the property that 
\begin{align*}T^i\circ R_\e = Q^i_\e \circ T^i, \;\;\;\;\;\;\;\; \mu\text{-a.e.}\;\;\;\;\;\; \text{ for all } \e \in (0,1].
\end{align*}
Assume that there exists some probability space $(\Omega,\mathcal F,P)$ which admits a coupling of the following form: there is a map $x:\Omega \to X$ such that the law of $x$ is $\mu$ (i.e., $P_*x = \mu$) and there is a collection of measurable functions $(\mathfrak J^1_\e,...,\mathfrak J^m_\e)_{\e\in(0,1]}$ from $\Omega \to \mathcal M$ such that 
\begin{equation}\label{ct3}\|\mathfrak J^i_\e(\omega) - Q^i_\e T^i(x(\omega))\|_{Y_i}\leq C^i(\omega)\end{equation}
for some collection of measurable functions $C^i:\Omega \to [0,\infty)$. Let $\mathcal M\subset Y_1\times \cdots Y_m$ be a closed subset, such that 
for all $\delta, \e>0$ 
$$P((\delta^{k_1}\mathfrak J^1_\e ,...,\delta^{k_m} \mathfrak J^m_\e) \in \mathcal M)=1, \;\;\text{ and }\;\; \mu(\{x\in X: (\delta^{k_1} T^1(x),...,\delta^{k_m} T^m(x)) \in \mathcal M\})=1.$$ Then the compact set $K:=\{(T^1_{\mathrm{hom}}(h),...,T^m_{\mathrm{hom}}(h)):h \in B(H)\}$ is necessarily contained in $\mathcal M$. Moreover, consider any Banach space $Z$ and any family of maps $\{\Phi_\e\}_{\e\in(0,1]}$ from $\mathcal M \to Z$ that converges uniformly on bounded subsets of $\mathcal M$ as $\e\to 0$ to a limit $\Phi$ that is uniformly continuous on bounded subsets of $\mathcal M$. 
Then the set of cluster points as $\e\to 0$ of $\Phi_\e\big( ((\log\log(1/\e))^{-k_i/2}\mathfrak J^i_\e(\omega))_{1\leq i\leq m}\big)$ is equal to $\Phi(K)$ for a.e. $\omega\in \Omega$.
\end{thm}

\begin{proof} Note that the assumptions are equivalent to the following statement: for all bounded sequences $\{\boldsymbol{y}_\e\},\{\boldsymbol{y}_\e'\}\subset \mathcal M$, we have \begin{equation}\label{dog}\lim_{\e\to 0}\|\Phi_\e(\boldsymbol{y}_\e) - \Phi(\boldsymbol{y}_\e')\|_Z= 0\;\;\;\text{ whenever }\;\;\;\lim_{\e\to 0}\|\boldsymbol{y}_\e-\boldsymbol{y}_\e'\|_{Y_1\oplus ... \oplus Y_m}= 0.\end{equation}Note that $(\log\log(1/\e))^{-k_i/2}Q^i_\e T^i(x(\omega))$ remains almost surely bounded as $\e\to 0$ by Theorem \ref{ar2}. Thus we have by \eqref{dog} that $$\big\| \Phi_\e\big( ((\log\log(1/\e))^{-k_i/2}\mathfrak J^i_\e(\omega))_{1\leq i\leq m}\big)-\Phi\big( ((\log\log(1/\e))^{-k_i/2}Q^i_\e T^i(x(\omega)))_{1\leq i\leq m}\big)\big\|_Z \to 0$$ as $\e\to 0$. The term on the right has the desired set of limit points by Theorem \ref{mainr}.
\end{proof}

Theorem \ref{contraction 3} is recovered by taking $\Omega=X$ and $\mathfrak J^i_\e = Q^i_\e J^i_\e$ and $\Phi_\e=\Phi$. Such a general statement would be useful in showing (for instance) that Strassen's law holds if we replace $\xi_\e$ by some \textit{non}-Gaussian approximation of a Gaussian noise, for instance a Poisson noise or some other discrete approximation of the desired limiting system, e.g., random walk approximations of Brownian motion, or statistical mechanical approximations of SPDEs such as KPZ. 

When applying the theorem in practice, one can usually obtain a stronger bound than \eqref{ct3}, typically an ``almost sure invariance principle" (ASIP) which gives a bound of the form $C^i(\omega)\e^\alpha$ where $\alpha>0$. For instance, in a random walk approximation of standard Brownian motion, the KMT coupling \cite{KMT} will give such a bound with $\alpha$ slightly less than 1/2 on the space constructed for the large-time regime in Example \ref{41}. For more complicated examples involving higher chaos, the discretized analogue $\mathfrak J^i_\e$ of each of the $m$ chaoses also needs to be constructed, and then the joint ASIP for these chaoses needs to be proved ``by hand." This is rather difficult in general, see for instance \cite{gmw} for some recent progress on a discrete approximation theory for Gaussian chaoses by general martingale chaoses. The mollification example above may be viewed a particular case where the ASIP for the chaoses is particularly simple to prove, compared to other choices of prelimiting objects such as discretizations.

\subsection{Iterated processes} In this subsection we consider a different generalization of the main result (Theorem \ref{mainr}) than those described in the previous subsection, specifically we study the effect of reparametrizing the $\e$ variable, which destroys the semigroup property. To motivate why one might be interested in such a reparametrization of $\e$, we consider iterated processes.

Burdzy in \cite{Bur93} proved that if $B,W$ are independent two-sided Brownian motions, then $$\limsup_{t\to 0} \frac{B(W(t))}{t^{1/4}\log(\log(1/t))^{3/4}}=2^{5/4}3^{-3/4}.$$ In \cite{CCFR95}, the authors extend this to prove a functional version of this theorem, namely that that the set of limit points of the family of random functions $Z^\epsilon :[-1,1]\to \mathbb R$ defined by \begin{equation}\label{zeps}Z^{\epsilon}(t):=\epsilon^{-1/4}(\log\log(1/\epsilon))^{-3/4}B(W(\epsilon t)),\end{equation} as $\epsilon\to 0$ equals the compact set of functions in $C[0,1]$ given by $$Q:=\{ f\circ g: f,g \in C([-1,1]), f(0)=g(0)=0, \int_{-1}^1 f'(t)^2+g'(t)^2 \leq 1\}.$$ Note that the composition is well-defined because the integral condition implies that $g$ maps $[-1,1]$ to itself.

We would like to recover this result using our abstract framework. For $X,Y\subset \mathbb R$ let $C(X,Y)$ denote continuous maps from $X$ to $Y$ and $CX := C(X,\mathbb R)$. Define $R_\e f(t) = \e^{-1/2}f(\e t).$ Note that the map $\Phi$ from $C[-2,2]\times C([-1,1],[-2,2]) \to C[-1,1]$ defined by $(f,g) \mapsto f\circ g$ is continuous. Furthermore we note that \begin{equation}\label{zeps2}Z^\e = \Phi\big( (\log\log(1/\e))^{-1/2} R_{(\e\log\log(1/\e))^{1/2}}B\;\;,\;\; (\log\log(1/\e))^{-1/2}R_\e W\big).\end{equation}
The latter expression makes sense because Lemmas \ref{glue} and \ref{ar} ensure that with probability $1$, $R_\e W$ has image contained in $[-2,2]$ for small $\epsilon$ almost surely.

Note that Theorem \ref{mainr} is not directly applicable here due to the expression $R_{(\e\log\log(1/\e))^{1/2}}B.$ This is of the form $R_{g(\e)}B$ where one still has that $\log(g(\e))/\log(\e)$ remains bounded away from $0$ and $+\infty$ as $\e\to 0$, but where $R_{g(\e)}$ does not respect the multiplicative semigroup property since $g(\e)g(\delta)\ne g(\e\delta)$. The main problem which we now face is that the variable $\e$ in the expression $R_\e W$ has \textit{not} been reparametrized in the same way. This motivates the following abstract result.

\begin{thm}\label{repar}
    For $1\leq j \leq r$ let $(\B_j,\H_j,\mu_j)$ be abstract Wiener spaces, let $(R^j_\e)_{\e\in(0,1]}$ ($1\le j \le r)$ be a family of Borel-measurable a.e. linear maps from $\B_j\to \B_j$ which are measure-preserving and strongly mixing, and let $T^{ij}:\B_j\to Y_{ij}$ be homogeneous of degree $k_{i,j}$ for $1\le i \le m$ and $1\leq j \leq r$, where $Y_{ij}$ are Banach spaces. Assume that $\mathcal M\subset \bigoplus_{i,j} Y_{ij}$ is a closed subset. Assume that there exist strongly continuous semigroups $Q^{ij}$ on $Y_{ij}$ such that one has $Q^{ij}_\e T^{ij} = T^{ij}R^j_\e$ $\mu_j$-a.e. for all $i,j$ and $\e\in (0,1]$. Also assume that for all $\delta>0$, $$(\mu_1\otimes \cdots \otimes \mu_r) \bigg( \big\{(x_1,...,x_r): \big(\delta^{k_{i,j}}T^{ij} (x_j)\big)_{i,j} \in \mathcal M \big\}\bigg)=1.$$ Then the compact set $K:= \{\big((T^{i,j})_{hom}(h_j)\big)_{i,j}: \|h_1\|_{H_1}^2+...+\|h_r\|_{H_r}^2\leq 1\}$ is contained in $\mathcal M$. Moreover consider any collection $\{g^j\}_{j=1}^r$ of maps from $(0,1]\to (0,1]$ such that $\log(g^j(\e))/\log(\e)$ remains bounded away from $0$ and $+\infty$ as $\e\to 0$. 
    Sample $(\xi_1,...,\xi_r)$ from $\mu_1 \otimes \cdots \otimes \mu_r$. Then for any Banach space $Z$ and any continuous function $\Phi: \mathcal M\to Z$, the set of cluster points as $\e\to 0$ of $\Phi\big(\big((\log\log(1/\e))^{-k_{i,j}/2}Q^{ij}_{g^j(\e)} T^{ij} (\xi_j)\big)_{i,j}\big)$ is a.s. equal to $\Phi(K).$
\end{thm}

Before the proof, we remark by \eqref{zeps2} that this theorem immediately gives the desired compact limit set for the iterated process $Z^\e$ from \eqref{zeps}, since $B$ and $W$ are assumed to be independent. Note that we may also obtain the result in every Hölder norm of exponent less than $1/4$ by noting that for all $0<\alpha<1$ the composition map is continuous from $C^\alpha[-2,2] \times C^\alpha([-1,1],[-2,2]) \to C^{\alpha^2}[-1,1]$, where $C^\alpha$ is the closure of smooth functions in the Hölder-$\alpha$ norm.

\begin{proof}
The proof essentially follows the same line of logic as the proof of Theorem \ref{mainr}. We will indicate the places in the argument that need to be modified in order for the theorem to be proved.

Firstly, in Lemma \ref{lem:lim1} we need weaker assumptions since we no longer have the structure of a stationary sequence. More precisely, thanks to the assumption that $\log(g^j(\e))/\log(\e)$ remains bounded, it suffices to show that if $(X_j)_{j\ge 0}$ is a (possibly non-stationary) sequence of joint mean-zero Gaussians of unit variance such that $\lim_{N\to \infty} \max_{|i-j|>N} |\mathbb E[X_iX_j]|=0 ,$ then one has that $\lim_{n\to\infty} (2\log n)^{-1/2}X_n=1.$ One may verify that this is indeed true by again using Slepian's lemma, alternatively see the main result of \cite{man}, where this was first observed.

This then leads to the following generalization of Corollary \ref{cor:sphere}. If $\mathbf X_n:=(X_n^{ij})_{1\leq i\leq N,1\leq j \leq r} \in \mathbb R^{N\times r} $ is a jointly Gaussian sequence such that the $r$ distinct collections given by $(X_n^{i1})_{i=1}^N , ..., (X_n^{ir})_{i=1}^N$ are independent of each other, and each satisfies $\mathrm{cov}(X_0^{ij},X_0^{i'j}) = \delta_{ii'}$ and $\lim_{A\to \infty} \max_{|m-n|>A} |\mathbb E[X_m^{ij}X_n^{i'j}]|=0$, then the unit sphere of $\mathbb R^{N\times r}$ is contained in the set of limit points of the sequence $((2\log n)^{-1/2} \mathbf X_n)_{n\ge 2}.$

From here, the line of proof will then use the same strategy as in Subsections 2.2 and 2.3, first proving the discrete-time formulation (see Theorems \ref{thm 4} and \ref{contraction}) along the subsequence $\e=2^{-n}$ by using the finite-rank approximations $T^{ij}_N$ from \eqref{frcmp'} and using the result of the previous paragraph. The remainder of the proofs are then straightforward to modify, but taking into account that there are now $r$ independent noises as opposed to just $1$, thus giving image of the closed unit ball of the Hilbert space $H_1\oplus ...\oplus H_r$ for the limit sets.
\end{proof}

\begin{example}
We now give another example using iterated processes derived from higher chaoses. This will recover the main result of \cite{Neu98}, where the author proves that if $W,B_1,B_2$ are independent two-sided standard Brownian motions, and if $A(t):=\frac12 \int_0^t B_2(s)dB_1(s)-B_1(s)dB_2(s), $ then $$\limsup_{t\to 0} \frac{A(W(t))}{t^{1/2}(\log\log(1/t))^{3/2}} = (2/3)^{-3/2}/\pi.$$
We can obtain a functional version of this using the above theorem. Define the process $Z^{\epsilon}(t):= \epsilon^{-1/2}(\log\log(1/\epsilon))^{-3/2} A(W(\epsilon t)),$ and note that $$Z^\epsilon:= \Phi ( (\log\log(1/\e))^{-1} Q_{(\e\log\log(1/\e))^{1/2}}A\;,\; (\log\log(1/\e))^{-1/2} R_\e W),$$ 
where $Q_\e f(t) = \e^{-1} f(\e t)$,  $R_\e f(t) = \e^{-1/2} f(\e t)$, and $\Phi(f,g)=f\circ g$ is the composition map as above. We remark that if we view $A=\psi(B_1,B_2)$  as a chaos of order 2, then we have a commutation relation $Q_\e \psi = \psi\tilde R_\e$ where $\tilde R_\e(f_1,f_2)(t) = (\e^{-1/2} f_1(\e t),\e^{-1/2} f_2(\e t))$ for $(f_1,f_2)\in C[0,1]\times C[0,1]$. Note in the above expression that the exponent of $-1$ on the double logarithm in the first coordinate is consistent with the fact that $A$ is a chaos of order 2.

Hence, mimicking the previous example but using the full generality of the above theorem, we see that $Z^\e$ has the set of limit points given by the closure of functions of the form $f \circ g$ where $f,g \in C[-1,1]$ are smooth and $f$ is of the form $\frac12\int_0^\bullet h_2h_1'-h_1 h_2'$ where $\int_{-1}^1 g'(t)^2 +h_1'(t)^2 +h_2'(t)^2dt\leq 1$, and $g(0)=h_i(0)=0$ for $i=1,2$.
\end{example}
\appendix

\section{Ergodic properties of measure-preserving linear operators}\label{AA}

In this appendix we review the fact that the ergodic properties of measure-preserving linear operators on an abstract Wiener space are necessarily determined by their action on the Cameron-Martin space.

\begin{defn}\label{ae} Let $(X,H,\mu)$ be an abstract Wiener space. We say that a Borel-measurable map $S:X\to X$ is a measure-preserving a.e. linear map if there exists a Borel-measurable linear subspace $E \subset X$ with $\mu(E)=1$ such that $S: E\to X$ is linear and $\mu(S^{-1}(F)) = \mu(F)$ for all Borel subsets $F \subset X.$

\end{defn}


\begin{lem}\label{bij}
Let $(X,H,\mu)$ be an abstract Wiener space. The set of measure-preserving a.e. linear maps (modulo a.e. equivalence) is in bijection with the set of bounded linear maps from $H\to H$ satisfying $SS^*=I$. Moreover the bijection is given by simply restricting $S$ to $H$.
\end{lem}

\begin{proof} We need to show that any a.e. defined measure-preserving linear transformation from $\B\to \B$ necessarily maps $\H\to \H$ boundedly and satisfies $SS^*=I$ on $\H$. To prove this, let $S:E \to \B$ be a measure preserving Borel measurable linear map, where $E\subset \B$ is a Borel measurable linear subspace satisfying $\mu(E)=1$. It is known \cite[Section 3.4]{Hai09} that $H$ equals the intersection of all Borel-measurable linear subspaces of $\B$ of measure $1$. If $F$ is any Borel-measurable linear subspace of measure $1$, then so is $E\cap S^{-1}(F)$,  and thus $x\in H$ implies $x\in S^{-1}(F)$ so that $Sx\in F$. Since $F$ is arbitrary, we have shown that $Sx\in H$ for all $x \in H$. Thus $S$ is a \textit{globally defined} Borel measurable linear map from $H\to H$, from which it follows that $S$ is automatically bounded \cite{Sch}. In order for $S$ to be measure-reserving, it must clearly satisfy $SS^*=I$ (e.g. by computing the covariance structure of the pushforward measure $S_*\mu$).

Conversely, 
given any bounded linear map $S:\H \to \H$ satisfying $SS^*=I$, there exists a $\mu$-a.e.-defined Borel-measurable linear extension $\hat S$ on $\B$ which is unique up to a.e. equivalence \cite[Sections 3.5 and 3.6]{Hai09}. Moreover the condition $SS^*=I$ guarantees that this extension is measure-preserving, as this is precisely the condition under which the covariance operator of the underlying Gaussian measure is unchanged.\end{proof}

Below and in the main body of the paper we always write $\hat S=S$ without specifying that it actually denotes the unique extension. Next we have a lemma about the structure of such measure-preserving transformations, namely that they can be orthogonally decomposed into a unitary part and a part converging strongly to zero.
\begin{lem}\label{easy}
Let $\H$ be a Hilbert space and consider any linear operator $S:\H\to\H$ satisfying $SS^*=I$. Then we can orthogonally decompose $H=A\oplus B$ where $A$ and $B$ are both invariant under $S$ and $S^*$. Moreover $S:A\to A$ is unitary and $\lim_{n \to \infty} \|S_nx\|_H = 0$ for all $x \in B.$ Explicitly one can write $A=\bigcap_{n \in \mathbb{N}} \Im(S_n^*)$ and $B=\overline{\bigcup_{n \in \N} \ker(S_n)}$, where the bar denotes the closure in $H$. 
\end{lem}
\begin{proof}
Since $S S^* = I$, we have $S_n S_n^* = I$.  It is clear that $S$ and $S^*$ both leave $\bigcap_{n \in \N} \Im(S_n^*)$ invariant and $S^* S$ is the identity there.
For each $n$, one has $H = \Im(S_n^*) \oplus \ker(S_{n})$ via the decomposition $x = S_n^*S_nx+ (x-S_n^*S_nx)$. Hence, $S_n^*S_n$ is merely the projection map from $H$ onto $\Im(S_n^*)$ and indeed the two given subspaces $A,B$ are orthogonal. Moreover it is clear that $\|S_n x\| \to 0$ on the closure of $\bigcup_{n \in \N} \ker(S_n)$, since $\|S_n x\|$ is eventually zero for all $x$ in the dense subspace $\bigcup \ker(S_n)$. \end{proof}

Below, we will refer to the subspace $A$ as the \textit{unitary part}. Given Lemma \ref{bij}, it is natural then to ask what conditions on $S$, when viewed as a map from $H\to H$, ensure that the measure-preserving extension $S$ is ergodic, weakly mixing and strongly mixing. 
The next proposition addresses this.
\begin{prop}\label{prop1}
Let $(\B,\H,\mu)$ be an abstract Wiener space. Consider any linear operator $S:\H \to \H$ satisfying $SS^*=I$. Then we have the following equivalences.
\begin{enumerate}[leftmargin = 2em]
    \item $\bigcap_{n \ge 1} \sigma(S_n)$ is a 0-1 $\sigma$-algebra (i.e., consists only of sets of $\mu$-measure 0 or 1) if and only if $\|S_nx\|_H\to 0$ as $n\to\infty$ for all $x \in H.$
    
    \item $S$ is strongly mixing if and only if $\langle S_nx,y\rangle_H \to 0$ as $n\to \infty$ for all $x,y\in H$.
    
    \item $S$ is ergodic if and only if any of the following five equivalent conditions hold:
\begin{enumerate}[leftmargin = 2em]

    \item
 $\frac1n \sum_{j=1}^n \langle S_j x,y\rangle_H^k \to 0$ as $n\to\infty$ for all $x,y\in \H$ and $k \in \mathbb N.$
    
    \item
    $\frac1n \sum_{j=1}^n |\langle S_j x,y\rangle_H| \to 0$ as $n\to \infty$ for all $x,y\in \H.$
    
\item
$S$ is weakly mixing.
\item
$S$ admits no invariant subspace of dimension two on which it acts by a rotation matrix. 
\item
The spectral measure $\mu_x$ of $S$ is atomless for every $x$ in the unitary part of $S$.


\end{enumerate}

\end{enumerate}
    
\end{prop}

These statements are classical, but we give a proof of the proposition for completeness. For example, Item \textit{(3e)} is a reformulation of the well-known Maruyama theorem on ergodicity of shifts of Gaussian fields \cite{Mar49}. The equivalence of ergodicity and weak mixing is a special case of e.g. \cite{RZ95}. Regarding \textit{(2)} and \textit{(3)}, Ustunel and Zakai 
have proved stronger results even for nonlinear maps of the Wiener space, see e.g. \cite{UZ93, UZ00, UZ01}. 

In the proof we will use the standard fact that for each $h \in H$ there exists an a.e. defined linear extension of the map from $H\to \mathbb R$ given by $v\mapsto \langle v,h\rangle_H$. By an abuse of notation we denote this linear extension as $\langle \cdot, h\rangle$ as well, and the map from $H\to L^2(X,\mu)$ given by $h\mapsto \langle h, \Cdot\rangle$ is a linear isometry. In particular, the law of each $\langle h, \Cdot \rangle$ is a Gaussian of variance $\|h\|_H^2$ with respect to $\mu$, see \cite[Section 3.4]{Hai09}.

\begin{proof}[Proof of Item (1)] Assume first that $\|S_n x\|_H\to 0$ for all $x \in H.$ Then one can decompose $H$ into an orthogonal direct sum: $H=\bigoplus_{n  \geq 0} H_n$ where $H_n:= \ker(S_{n+1})\cap \Im(S_n^*),$ via the formula $$x = \sum_{k=0}^{\infty} S_k^*S_kx - S_{k+1}^*S_{k+1}x. $$The series converges to $x$ in $H$ because the $N^{th}$ partial sum equals $x-S_{N+1}^*S_{N+1}x$ and we know that $\|S_{N+1}^*S_{N+1}x\|_H \leq \|S_{N+1}x\|_H \to 0$. Moreover one easily checks that $S_k^*S_kx - S_{k+1}^*S_{k+1}x \in \ker(S_{k+1}) \cap \Im (S_{k}^*)$, and that $\ker(S_i) \cap \Im(S_{i-1}^*)$ is orthogonal to $\ker(S_j) \cap \Im (S_{j-1}^*)$ for $i<j$. 

Let $\xi$ be a $B$-valued random variable with law $\mu$ and define $\xi_n$ to be the projection of $\xi$ onto $H_n$. Note that the $\xi_n$ are independent $\B$-valued random variables (but not necessarily i.i.d.). Note also that $S_n\xi$ is measurable with respect to $\{\xi_j\}_{j \ge n}$. Consequently $\bigcap_{n \geq 1} \sigma(S_n\xi) 
\subseteq
\bigcap_{n \geq 1} \sigma(\{\xi_j:j\ge n\}),$ which by Kolmogorov's 0-1 law is a 0-1 sigma algebra.

Conversely, suppose that $\|S_nx\|\not\to 0$ as $n \to \infty$ for some $x\in \H.$ Then by Lemma \ref{easy}, the closed subspace $A:=\bigcap_{n \geq 0} \Im(S_n^*)$ is nonzero. Letting $\xi$ denote a random variable in $\B$ with law $\mu$, let $\xi_A$ denote the projection onto $A$ applied to $\xi$. Since $S|_A$ is unitary (again by Lemma \ref{easy}) it is clear that $\bigcap_n \sigma(S_n)$ contains at least $\sigma(\xi_A)$, which is nontrivial since $A\neq\{0\}$ so that at least one nonzero Gaussian random variable is measurable with respect to it.
\end{proof}

\begin{proof}[Proof of Item (2)] Suppose that $\langle S_nx,y\rangle_{H}\to 0$ for all $x,y\in \H$. We wish to show that  
\begin{equation*}
\lim_{n \to \infty} \int_\B f(S_nx)g(x)\mu(dx)= 0 
\end{equation*}
for all bounded measurable functions $f,g:\B\to\mathbb R$ such that $\int_\B fd\mu=\int_\B gd\mu = 0$. By an application of Cauchy-Schwarz and the measure-preserving property of $S_n$, it suffices to prove this in a dense subspace of $L^2(\B,\mu)$.

By using the Wick forumla (aka Isserlis' theorem), one can easily show that the claim is at least true whenever $f$ and $g$ are both of the form $x \mapsto p(\langle x,e_1\rangle,...,\langle x,e_k\rangle)$ for some $k\in \mathbb N$, some polynomial $p:\mathbb R^k\to \mathbb R$ and some orthonormal set of vectors $e_1,...,e_k$ in $H$. Then by a density argument and the fact that the $\langle x,e_i\rangle$ has Gaussian tails, one can extend this from $k$-variable polynomials $p$ to all continuous functions $f,g:\mathbb R^k \to \mathbb R$ with at-worst polynomial growth at infinity. 
One can then extend to all bounded Borel-measurable functions $f,g: \mathbb R^k\to \mathbb R$ by density of continuous functions in $L^2(\mathbb R^k,\gamma_k)$ where $\gamma_k$ is the standard Gaussian measure on $\mathbb R^k$.

Thus, to finish the argument, it suffices to show that the set of all functions of the form $x\mapsto f(\langle x,e_1\rangle,...,\langle x,e_k\rangle)$, where $f:\mathbb R^k\to \mathbb R$ is bounded and measurable, is dense in $L^2(\B,\mu)$. To show this, choose an orthonormal basis $\{e_j\}$ for $H$ and Let $\mathcal F_n$ denote the sigma algebra generated by $\langle x,e_1\rangle,...,\langle x, e_n\rangle$. Let $f\in L^{\infty}(\B,\mu)$ and let $f_n:= \mathbb E[f|\mathcal F_n].$ Then $f_n$ is bounded and measurable and of the form $x\mapsto h(\langle x,e_1\rangle,...,\langle x,e_n\rangle)$. Furthermore by martingale convergence $\|f_n-f\|_{L^2}\to 0$, completing the proof.

The converse direction is striaghtforward: if $S$ is strongly mixing, then apply the mixing definition to $f(x) := \langle x, a\rangle $ and $g(x) := \langle x, b\rangle$ to conclude that $\langle S_na,b\rangle\to 0$.
\end{proof}

\begin{proof}[Proof of Item (3)] We are going to show that ergodicity implies (a) which implies (b) which implies (c). Clearly (c) implies ergodicity. Then we show that (d) is equivalent to (e) for the unitary part of the operator from the decomposition in Lemma \ref{easy}. Then we will show that for unitary operators, (e) holds if and only if (b) holds.

So assume that $S$ is ergodic as a map from $\B\to\B.$ This is equivalent to the statement that $\int_\B \big( \frac1n \sum_{j = 1}^n f(S_jx)\big) g(x) \mu(dx) \to 0$ for all $f,g \in L^2(\B,\mu)$ such that $\int_\B fd\mu=\int_\B gd\mu = 0$. Now fix $k\in\mathbb N$ and $a,b \in H$. Letting $H_k$ denote the $k^{th}$ Hermite polynomial, we set $f(x) := \frac1{\sqrt{k!}}H_k(\langle x,a\rangle)$ and $g(x):= \frac1{\sqrt{k!}}H_k(\langle x , b\rangle).$ Then $\int_\B f(S_jx)g(x)d\mu = \langle a,S_j b \rangle_H^k$ (see e.g. \cite{Nua96}), so by ergodicity we obtain (a).

Now assume (a) holds. Let $a,b\in H$ with $\|a\|=\|b\|=1$. Fix $\epsilon>0$ and let $p:[-1,1]\to\mathbb R$ be a polynomial such that $\sup_{x\in[-1,1]} \big|p(x)-|x|\big|<\epsilon.$ Since $|\langle a,S_j b\rangle|\leq 1$ for all $j \in \mathbb N$, it follows that $\frac1n\sum_{j = 1}^n \big|p(\langle a,S_jb\rangle) -|\langle a,S_jb\rangle|\big|<\epsilon.$ Moreover, since (a) holds we know that $\frac1n \sum_{j = 1}^n p(\langle a,S_jb\rangle)\to 0$. Consequently we find that $\limsup_n \frac1n \sum_{j = 1}^n |\langle a,S_jb\rangle_H| \leq \epsilon.$ Since $\epsilon$ is arbitrary, it follows that (b) holds.

Now assume that (b) holds. To show (c), we want that $\frac1n \sum_{j = 1}^n \big| \int_\B f(S_jx)g(x)\mu(dx)\big| \to 0$ for all $f,g \in L^2(\B,\mu)$ that have mean zero. First note that, by essentially the same series of density arguments given in the proof of Item \textit{(2)}, it suffices to prove this whenever $f,g$ are of the form $x \mapsto p(\langle x,e_1\rangle,...,\langle x,e_k\rangle)$ for some $k\in \mathbb N$, some polynomial $p:\mathbb R^k\to \mathbb R$ and some orthonormal set of vectors $e_1,...,e_k$ in $H$. In turn, by Wick formula it suffices to show that $\frac1n\sum_{j=1}^n |\langle a,S_j b\rangle_H|^k \to 0$ for all $a,b\in H$ and all $k \in \mathbb N$. 
Note that we have $|\langle a,S_jb\rangle|\leq \|a\|\|b\|,$ so $|\langle a,S_jb\rangle_H|^k \leq \|a\|^{k-1}\|b\|^{k-1} |\langle a,S_jb\rangle_H|$.  Summing over $j$ and applying (b), we conclude $\frac1n\sum_{j=1}^n |\langle a,S_j b\rangle_H|^k \to 0$.

It is straightforward to show that (using the spectral theorem of the unitary operator) atoms of $\mu_x$ correspond precisely to complex eigenvalues of $S$, i.e., two-dimensional subspaces on which $S$ acts by rotation. Thus (d) implies (e) and vice versa (by focusing only on the 
Finally we explain why (e) is equivalent to (b). 
The spectral measure $\mu_x$ of $S$ is supported on $\mathbb T:=\{z\in\mathbb C: |z|=1\}$ and defined via its Fourier transform: $\hat \mu_x (k):= \langle S_kx,x\rangle$ where $k\in \mathbb Z$ and $S_k:= S_{-k}^*$ if $k<0$. Thus, to show that (b) and (e) are equivalent, we just need to show that a finite measure $\mu$ on $\mathbb T$ is atomless if and only if $\frac1n \sum_{k=-n}^n |\hat \mu(k)| \to 0$ as $n \to \infty$. This is a direct consequence of Wiener's Lemma.
\end{proof}

\section{Gaussian measures on spaces of distributions}

The results of this appendix are included only for completeness and convenience. Here we show under general conditions how to identify the Cameron-Martin space of a given Gaussian measure on some space of distributions, and how to identify when a semigroup of operators on some Banach space of distributions is strongly continuous or satisfies the mixing condition, as needed in Theorem \ref{mainr}. We let $U\subset \mathbb R^d$ be any open set (possibly unbounded), and we let $\bar U$ denote its closure.

\begin{defn}\label{b1} We shall consider a vector space $\mathcal F(U)$ of smooth bounded functions $f:\bar U\to \mathbb R$ (not necessarily vanishing on $\partial U$) such that all of the seminorms $$\|f\|_{\beta,n }:= \sup_{x\in \bar U} (1+|x|^{2n})|\partial^\beta f(x)|$$ are finite, as we range over $n\in \mathbb N$ and multi-indices $\beta.$ Here $|x|$ denotes Euclidean norm of the vector $x\in \mathbb R^d.$ We will always assume the following:
\begin{itemize}
    \item $\mathcal F(U)$ contains all smooth functions with compact support contained in $U$.

    \item $\mathcal F(U)$ is complete with respect to the above family of seminorms.
\end{itemize}
Such a space $\mathcal F(U)$ has the structure of a Frechet space if we endow it with these seminorms, hence it has a continuous dual which will be denoted as $\mathcal F'(U)$, endowed with its weak* topology.
\end{defn}

Note that one has a natural embedding $\mathcal F(U) \hookrightarrow \mathcal F'(U)$ given by identifying a smooth function $f$ with the element in the dual given by sending $g \mapsto \int_{\bar U} fg$. Such an embedding is injective by the first bullet point above. Below we will not distinguish $\mathcal F(U)$ from the image of this embedding, instead simply viewing all spaces as subsets of $\mathcal F'(U)$. 

A typical example of $\mathcal F(U)$ will be of the following form: assume $U\subset \mathbb R^d$ is bounded with smooth boundary, then partition the boundary of $U$ into some finite number of subsets, and let $\mathcal F(U)$ be the set of smooth functions whose derivatives satisfy some linear relation on each element of the partition (e.g. $U=(0,1)$ and $\mathcal F(U)$ consists of those smooth functions on $[0,1]$ whose derivatives up to order $k$ vanish at the left endpoint, where $k\in \mathbb N$). 
Unless $U=\mathbb R^d$, the spaces $\mathcal F(U)$ and $\mathcal F'(U)$ will typically \textit{not} be closed under the partial differentiation operators in the same way that $\mathcal S(\mathbb R^d)$ and $\mathcal S'(\mathbb R^d)$ are.

Note that the continuous dual $\mathcal F'(U)$ as a topological space (despite not being second-countable) has a Borel $\sigma$-algebra which is generated by a countably infinite set of continuous linear functionals $\mathcal F'(U)\to \mathbb R$. Indeed continuous linear functionals on $\mathcal F'(U)$ are precisely elements of $\mathcal F(U)$. The latter is second-countable and metrizable, hence admits a countable dense subset. Since the topology of $\mathcal F'(U)$ is (by definition) generated by all continuous linear functionals on it, this countable set of continuous linear functionals is sufficient to generate the Borel $\sigma$-algebra on $\mathcal F'(U)$.


\begin{defn}\label{emb}We say that a Banach space $X$ is \textit{embedded
between} $\mathcal F(U)$ and $\mathcal F'(U)$ if $X$ is a Borel-measurable linear subspace of $\mathcal F'(U)$ which contains $\mathcal F(U),$ and such that the norm of $X$ (defined to be infinity outside of $X$) is a Borel-measurable function on $\mathcal F'(U)$.
\end{defn}

The norm topology on a Banach space is the one induced by the metric $d(x,y)=\|x-y\|$. The weak topology on a Banach space $X$ is the one generated by the linear functionals $f\in X^*$, and the weak* topology on $X^*$ is the one generated by the linear functionals on $X^*$ given by $\hat x: f\mapsto f(x)$ as $x$ ranges throughout $X.$ 
\begin{lem}\label{b1a}
    A linear subspace of a Banach space is weak-dense if and only if it is norm-dense.
\end{lem}

The proof is immediate by Hahn-Banach. Our next lemma is a result which roughly says that if $X$ is a Banach space of distributions that is embedded between $\mathcal F(U)$ and $\mathcal F'(U)$ in such a way that $X$ densely contains $\mathcal F(U)$, then its dual space $X^*$ is also embedded between $\mathcal F(U)$ and $\mathcal F'(U)$, up to an isometric isomorphism. The dual space may no longer densely contain $\mathcal F(U)$, but in the weak* topology it still does. We will use $(\cdot,\cdot)$ to denote the natural pairing between $\mathcal F(U)$ and $\mathcal F'(U)$ with $(\phi,\psi)=\int_{\bar U} \phi \psi $ if $\phi,\psi \in \mathcal F(U)$, i.e., the pairing in $L^2(\bar U)$.

\begin{lem}\label{xdu}
Suppose $X$ is a Banach space that is embedded between $\mathcal F(U)$ and $\mathcal F'(U)$. Suppose that the image of the first embedding is norm-dense in $X$. Then there exists a Banach space $X^{du}$ which is also embedded between $\mathcal F(U)$ and $\mathcal F'(U)$, and which admits a "canonical" bilinear pairing $B:X\times X^{du} \to \mathbb R$ in the sense that following properties hold: \begin{enumerate}[leftmargin = 2em]

    \item $|B(x,f)| \le \|x\|_X\|f\|_{X^{du}}$.

    \item The map from $X^{du}\to X^*$ given by $f\mapsto B(\cdot,f)$ is an isomorphism and a linear isometry.
    
    \item $B(\phi,f)=(f,\phi)$ and $B(x,\phi) = (x,\phi)$ for all $x\in X$, all $f\in X^{du},$ and all $\phi\in \mathcal F(U).$
\end{enumerate}
The image of the first embedding $\mathcal F(U) \hookrightarrow \B^{du}$ may \textbf{not} be norm-dense in $X^{du}$, but it is always dense with respect to the weak* topology on $X^{du}$.
\end{lem}

Before proving the lemma, let us give examples when $U=\mathbb R^d$ so $\mathcal F(U) = \mathcal S(\mathbb R^d)$. If $X=L^p(\mathbb R^d)$ with $1\le p<\infty$, then $X^{du}=L^q(\mathbb R^d)$ with $\frac1p+\frac1q=1.$ If $X=C_0(\mathbb R^d)$ then $X^{du}$ consists of finite signed Borel measures on $\mathbb R^d$ equipped with total variation norm. Note in this case that the weak-closure of $\mathcal S(\mathbb R^d)$ in $X^{du}$ is $L^1(\mathbb R^d)$, which is a proper closed subspace of $X^{du}$, but still dense with respect to weak (i.e. Prohorov) convergence of measures. Further examples are given by Sobolev spaces $X=H^s(\mathbb R^d)$ and $X^{du}=H^{-s}(\mathbb R^d)$ with $s\in \mathbb R$. A rich class of examples is given by more general Sobolev spaces and Besov spaces including those defined with weight functions.

\begin{proof}
Note that we have a map $\mathcal G$ from $X^*\to \mathcal F'(U)$ given by restriction to the Schwartz class, i.e., $f \stackrel{\mathcal G}{\mapsto} f|_{\mathcal F(U)}.$ The map clearly defines a continuous linear operator from $X^*\to \mathcal F'(U)$. We define $X^{du}$ to be the image of $X^*$ under $\mathcal G$. We also note that $\mathcal G$ is injective since $\mathcal F(U)$ is dense in $X$ by assumption. We thus define the norm on $X^{du}$ by $\|\mathcal Gf\|_{X^{du}} = \|f\|_{X^*}.$ Then clearly $X^{du}$ is a Banach space that is isometric to $X^*$ (via $\mathcal G$) and the inclusion $X^{du}\hookrightarrow \mathcal F'(U)$ is clearly continuous. 

We now claim that $\mathcal F(U)$ is contained in $X^{du}$. To prove this we need to check that if $\phi\in \mathcal F(U)$ then the map $a_{\phi}$ from $X\to \mathbb R$ given by $x \mapsto (x,\phi)$ is continuous. This is clear because if $x_n\to x$ in $X$, then $x_n\to x$ in $\mathcal F'(U)$ so that $(x_n,\phi)\to (x,\phi).$ Next, we note that $\phi = \mathcal Ga_{\phi} \in X^{du}$, proving the claim. By the closed graph theorem, it follows that the inclusion map $\mathcal F(U) \hookrightarrow X^{du}$ is automatically continuous.

Now we construct the bilinear map $B$. For this, we simply define $B(x,\mathcal G f):=f(x)$ whenever $x\in X$ and $f\in X^*$. Clearly $|B(x,\mathcal Gf)| \leq \|f\|_{X^*}\|x\|_X = \|\mathcal Gf\|_{X^{du}}\|x\|_X,$ so that $B$ is bounded. Note that if $\phi\in\mathcal S(\mathbb R^d)$ then $B(\phi,\mathcal Gf) = f(\phi)= (f|_{\mathcal S(\mathbb R^d)},\phi) = (\mathcal Gf,\phi)$, as desired. Also $B(x,\phi) = B(x,\mathcal Ga_{\phi}) = a_{\phi}(x) = (x,\phi),$ completing the proof. The map $f\mapsto B(\cdot,f)$ is an isometry from $X^{du}\to X^*$ because it is inverse to the isometry $\mathcal G.$

Finally we need to show that $\mathcal F(U)$ is weak* dense in $X^{du}$. This follows immediately from the fact that $\mathcal F(U)$ is a total set in $X^{du}$, i.e., it separates points of $X$. This is clear because it separates points of the larger space $\mathcal S'$. It is known that totality is equivalent to weak* density of the finite linear span of any given subset of $X^*$, see \cite[Corollary 5.108]{charalambos2013infinite}.
\end{proof}

\begin{rk}\label{3.2}The assumptions in Lemma \ref{xdu} automatically imply separability of $X$ in its norm topology, because $\mathcal F(U)$ is necessarily second-countable and completely metrizable. However, we also remark that a linear subspace of $X^*$ can certainly be weak* dense without being norm dense, e.g., $C_0(\mathbb R)$ is dense in the weak* topology but not the norm topology of $L^\infty(\mathbb R) = L^1(\mathbb R)^*.$ Consequently under the assumptions of Lemma \ref{xdu}, $X^{du}$ may not be separable in its norm topology (e.g. take $U = \mathbb R$ and take $X$ to be $L^1(\mathbb R)$ or $C_0(\mathbb R)$).
\end{rk}
\begin{rk} In the above proof, we stated the formula $B(x,\mathcal Gf)=f(x)$. This implies that for $f\in X^{du}$, the dual norm is given by the conjugate formula\begin{equation}\label{xdu2}\|f\|_{X^{du}} = \sup\{ (f,\phi): \phi\in \mathcal F(U), \|\phi\|_{X}\leq 1\},
\end{equation}
since $(f,\phi)=B(f,\phi)$ if $\phi\in \mathcal F(U)$. Conversely, if $f\in \mathcal F'(U)$ such that the right side is finite, then $f\in X^{du}$ with norm given by that supremum, since $(f,\cdot)\in X^*$ so that $\mathcal G((f,\cdot))=f$.
\end{rk}

A corollary Lemma \ref{xdu} is that if we want to check that a given Hilbert space $H$ is the Cameron-Martin space of a Gaussian measure $\mu$ on some separable Banach space $X$ of distributions as above, then it suffices to check the action of the covariance function only on smooth test functions, as opposed to the entire dual space of $X$.

\begin{lem}\label{dualnorm}
Let $H$ be a Hilbert space and let $X$ be a Banach space such that both spaces are embedded between $\mathcal F(U)$ and $\mathcal F'(U). $ Suppose that $\mathcal F(U)$ is norm-dense in both $H$ and $X$. Let $\mu$ be a Gaussian measure on $X$. If \begin{equation}\label{assn1}\int_X (x,\phi)(x,\psi)\mu(dx) = \langle \phi,\psi \rangle_H\end{equation} for all $\phi,\psi \in \mathcal F(U),$ then $\mu$ has Cameron-Martin space $H^{du},$ which is necessarily contained in $X$.
\end{lem}

\begin{proof}
The Cameron-Martin norm may be defined for $h\in X$ by 
\begin{equation}\label{eq:CM}
\|h\|_{CM}=\sup\{f(h): f\in X^*, \int_X f^2d\mu\le 1\}.
\end{equation}
By the separability of $X$ and weak* density of $\mathcal F(U)$ in $X^{du}$, for any $f \in X^*$, there exist  $f_n \in \mathcal F(U)$ such that $f_n$ converges to $f$ pointwise. Since for a sequence of Gaussian random variables, pointwise convergence implies the convergence in variance, hence we have $
    \lim_{n \to \infty} \int_X f_n^2 d\mu = \int_X f^2 d\mu.
$
Using this together with \eqref{eq:CM}, we have
\begin{align*}
\|h\|_{CM} &= \sup\{(h,\phi): \phi \in \mathcal F(U), \int_X (x,\phi)^2 \mu(dx)\le 1\}\\
&= \sup\{(h,\phi): \phi \in \mathcal F(U), \|\phi\|_H^2\le 1\}.
\end{align*}
The second equality is due to \eqref{assn1}.  By the weak density of $\mathcal S$ in $H$, the right hand side above equals the operator norm of the linear functional on $H$ given by $(\cdot,h)$, which by \eqref{xdu2} equals the conjugate norm $\|h\|_{H^{du}}$ (to be understood as $+\infty$ if $h\notin H^{du}$). Hence, for $h \in X \cap H^{du}$, we have $\|h\|_{CM} = \|h\|_{H^{du}}$.

So far this argument shows that the Cameron-Martin space of $\mu$ equals $X \cap H^{du}$ (with the norm of $H^{du}),$ which is therefore closed in $H^{du}$. To finish the proof we need to show that $H^{du}$ does not contain any vectors outside of $X$. Assume for contradiction that such a vector does exist; then $X\cap H^{du}$ would have a nonzero orthogonal complement with respect to the inner product of $H^{du}$. Take some nonzero bounded linear functional $u: H^{du} \to \mathbb R$ which vanishes on the closed subspace $X\cap H^{du}$. Since $H$ is reflexive, every linear functional on $H^{du}$ is represented as $B(f,\cdot)$ for some $f\in H$ where $B$ is the bilinear form constructed in the previous lemma. Thus write $u=B(f,\cdot)$ for some $f\in H$. Since $X\cap H^{du} $ contains $\mathcal F(U)$ we see that $(f,\phi) = B(f,\phi) = 0$ for all $\phi \in \mathcal F(U)$. This means that $f=0$ so that $u=0,$ a contradiction.
\end{proof}

\begin{lem}\label{opdu}Suppose that $Q$ is a bounded linear operator from $L^2(\bar U)\to L^2(\bar U)$ such that $(Q\phi,\phi)>0$ for all $\phi\in \mathcal F(U).$ 
Let $X$ be a Banach space that is embedded between $\mathcal F(U)$ and $\mathcal F'(U)$. Assume that \begin{itemize}\item $\mu$ is a Gaussian measure on $X$ such that $\int_X (x,\phi)(x,\psi)\mu(dx) = (Q\phi,\psi)$ for all $\phi,\psi\in\mathcal F(U)$. \item The completion of $L^2(\bar U)$ with respect to $(Qf,f)^{1/2}$ embeds continuously into $\mathcal F'(U)$.
\end{itemize}
Then $Q$ has a unique positive square root $Q^{1/2}$, and the image of $L^2(\bar U)$ under $Q^{1/2}$ contains $\mathcal F(U)$, and Cameron-Martin space of $\mu$ is the completion of $\mathcal F(U)$ under the norm $\|Q^{-1/2}\phi\|_{L^2(\bar U)}.$
\end{lem}
We remark that the second bullet point is a nontrivial condition. For example if $U=\mathbb R^d$ and $Q$ denotes convolution with heat kernel at time 1, then the image of $Q^{1/2}$ does \textit{not} contain $\mathcal F(U) = \mathcal S(\mathbb R^d).$ As always $(\cdot,\cdot)$ here denotes the pairing in $L^2(\bar U).$
\begin{proof}First let us verify that $Q$ is closable and positive-definite on $L^2(\bar U)$. The assumption that that $\int_X (x,\phi)(x,\psi)\mu(dx) = (Q\phi,\psi)$ guarantees that $Q$ is self-adjoint. 
It is positive-definite since we assumed that $(Q\phi,\phi)>0$ for all $\phi$ in the dense subspace $\mathcal F(U)$ of $L^2(\bar U)$. By functional calculus (or just spectral theorem) we can define $Q^\alpha$ for any $\alpha\in\mathbb R$, and this is also a closable and densely defined operator, such that the domain of $Q^\alpha$ is contained in the domain of $Q^\beta$ if $\alpha<\beta<0$. Let $\mathcal I$ denote the domain of $Q^{-1}$. 

Let $H$ be the Hilbert space obtained by completing $\mathcal F(U)$ with respect to the norm $(Q\phi,\phi)^{1/2}.$ By assumption $H$ embeds into the continuous dual space $\mathcal F'(U)$. 
Then for $\phi \in \mathcal I$ one has \begin{align*}\|Q^{-1/2}\phi\|_{L^2(\bar U)}=\sqrt{(Q^{-1}\phi,\phi)} &= \sup\{ (\phi,\psi): \psi\in \mathcal F(U), (Q\psi,\psi)\leq 1\} \\&= \sup\{ (\phi,\psi): \psi\in \mathcal F(U), \|\psi\|_{H}\leq 1\} = \|\phi\|_{H^{du}},
\end{align*}
where the last equality follows from \eqref{xdu2}. Consequently we may conclude that $H^{du}$ is simply the completion of $\mathcal I$ equipped with the norm given by $\|Q^{-1/2}\phi\|_{L^2(\bar U)},$ and by Lemma \ref{xdu} the latter Hilbert space is in turn necessarily embedded between $\mathcal F(U)$ and $\mathcal F'(U)$ in the sense of Definition \ref{emb}. Lemma \ref{xdu} also implies that $\mathcal F(U)$ is dense in the weak* topology, but as $H$ is a Hilbert space, the weak and weak* topology coincide. As $\mathcal F(U)$ is a linear subspace, Lemma \ref{b1a} implies it is norm dense.
\end{proof}

Next we identify conditions for a semigroup $S_t$ of operators on a Banach space $X$ embedded in $\mathcal F'(U)$ to be strongly continuous. The following lemma is straightforward given Remark \ref{3.2}. 

\begin{lem}\label{sg}
Let $X$ be a Banach space that is embedded between $\mathcal F(U)$ and $ \mathcal F'(U)$, such that $\mathcal F(U)$ is weakly dense in $X$. Let $S_t:X\to X$ be a semigroup of bounded operators for $t\ge 0$ such that 
\begin{enumerate}
\item each $S_t$ maps $\mathcal F(U)$ to itself.
\item $S_t \phi \stackrel{\mathcal F(U)}{\to} \phi$ as $t\to 0$ for all $\phi \in \mathcal F(U)$.
\item $\sup_{t\leq 1}\|S_t\|_{X\to X} <\infty.$
\end{enumerate} 
Then $(S_t)$ is a strongly continuous semigroup on $X$.
\end{lem}

In the main body of the paper, we often use the above lemma in multiplicative form, where the assumptions translate to $R_\e\phi\stackrel{\mathcal F(U)}{\to} \phi$ as $\e\uparrow 1$ and $\sup_{\e\in[e^{-1},1]}\|R_\e\|_{X\to X}<\infty$. Finally the following lemma gives an easy way to check the mixing condition which is required in Theorem \ref{mainr}.

\begin{lem}\label{sg'}
Let $H$ be a Hilbert space that is embedded between $\mathcal F(U)$ and $ \mathcal F'(U)$, such that $\mathcal F(U)$ is weakly dense in $H$. Let $\mathcal F^o(U)$ be a dense subset of $\mathcal F(U)$. Let $S:H\to H$ be a bounded operator such that $SS^*=I$ and $\langle S^n \phi, \psi\rangle_H\to 0$ for all $\phi,\psi \in \mathcal F^o(U)$. Then $\langle S^n \phi, \psi\rangle_H\to 0$ for all $\phi,\psi \in H.$
\end{lem}

Since $\mathcal F^o(U)$ is necessarily norm-dense in $H$ (Remark \ref{3.2}), the proof is clear by a density argument and the fact that $\|S^n\|_{H\to H}\leq 1$ for all $n$. The lemma makes the mixing condition easy to check, since in examples of interest we often have that $\mathcal F^o(U)$ consists of compactly supported functions, and often $\langle S^n \phi, \psi\rangle_H= 0$ for large enough $n$ whenever $\phi,\psi$ are compactly supported.

\bibliographystyle{alpha}
\bibliography{ref.bib}

\end{document}